\documentclass[11pt]{amsart}

\usepackage[marginratio=1:1,height=660pt,width=480pt,tmargin=70pt]{geometry}

\usepackage{amssymb,latexsym,amsmath,amsthm,amscd}
\usepackage{mathrsfs}   
\usepackage{fancyhdr}
 \usepackage{ifthen} 
 \usepackage{xstring}

\usepackage{etoolbox}

\newcommand{\im}{\lrcorner\,}

\appto\appendix{\addtocontents{toc}{\protect\setcounter{tocdepth}{1}}}
\appto\listoffigures{\addtocontents{lof}{\protect\setcounter{tocdepth}{1}}}
\appto\listoftables{\addtocontents{lot}{\protect\setcounter{tocdepth}{1}}}
\theoremstyle{plain}
\newtheorem{theorem}{Theorem}[section]

\newtheorem{proposition}[theorem]{Proposition}
\newtheorem{definition}[theorem]{Definition}

\theoremstyle{remark}

\newtheorem{remark}[theorem]{Remark}

\newtheorem*{example*}{Example}

\numberwithin{equation}{section}

\usepackage{enumitem,array,rotating}  
\usepackage{color}
\usepackage{xcolor}
\definecolor{carmine}{rgb}{0.59, 0.0, 0.09}
\definecolor{mediumpersianblue}{rgb}{0.0, 0.4, 0.65}
\definecolor{persianplum}{rgb}{0.44, 0.11, 0.11}
\usepackage[colorlinks=true,
            linkcolor=persianplum, 
            urlcolor=olive,
            citecolor=mediumpersianblue, 
            backref=page]{hyperref}
\usepackage{amsmath}
\usepackage{amsfonts}
\usepackage{amssymb}
 
\newcommand{\cE}{\mathcal{E}}
\newcommand{\cG}{\mathcal{G}}

\newcommand{\cP}{\mathcal{P}}

\newcommand{\cI}{\mathcal{I}}

\newcommand{\scX}{\mathscr{X}}
\newcommand{\scV}{\mathscr{V}}
\newcommand{\scD}{\mathscr{D}}

\newcommand{\scC}{\mathscr{C}}
\newcommand{\scH}{\mathscr{H}}

\newcommand{\fg}{\mathfrak{g}}
\newcommand{\fp}{\mathfrak{p}}

\newcommand{\fs}{\mathfrak{s}}
\newcommand{\ri}{\mathrm{i}}

\newcommand{\rD}{\mathrm{D}}

\newcommand{\R}{\mathbb{R}}
\newcommand{\RR}{\mathbb{R}}
\newcommand{\KK}{\mathbb{K}}

\newcommand{\CC}{\mathbb{C}}

\newcommand{\PP}{\mathbb{P}}
\newcommand{\QQ}{\mathbb{Q}}
\newcommand{\HH}{\mathbb{H}}
\newcommand{\SSS}{\mathbb{S}}

\newcommand{\bC}{\mathbf{C}}

\newcommand{\bH}{\mathbf{H}}

\newcommand{\bP}{\mathbf{P}}
\newcommand{\bQ}{\mathbf{Q}}
\newcommand{\bR}{\mathbf{R}}

\newcommand{\bT}{\mathbf{T}}

\newcommand{\w}{{\,{\wedge}\;}}

\newcommand{\exd}{\mathrm{d}}

\newcommand{\ve}{\varepsilon}

\newcommand{\spn}{\operatorname{span}}

\newcommand{\half}{\textstyle{\frac 12}}
\newcommand{\Dt}{\textstyle{\frac{\mathrm D}{\exd t}}}

\makeatletter
\renewcommand*{\p@section}{\S\,}
\renewcommand*{\p@subsection}{\S\,}
\renewcommand*{\p@subsubsection}{\S\,}
 
\makeatother
\usepackage{float}

\begin{document}

\author{Wojciech Kry\'nski}\author{Omid Makhmali}

 \address{\newline Wojciech Kry\'nski\\\newline
   Institute of Mathematics, Polish Academy of Sciences, ul.~\'Sniadeckich 8, 00-656 Warszawa, Poland\\\newline
   \textit{Email address: }{\href{mailto:krynski@impan.pl}{\texttt{krynski@impan.pl}}}\\\newline\newline
   Omid Makhmali\\\newline
   Departamento de Geometr\'{\i}a y Topolog\'{\i}a and IMAG, Universidad de Granada, Granada 18071, Spain \\\newline
   \textit{Email address: }{\href{mailto:omakhmali@ugr.es}{\texttt{omakhmali@ugr.es}}}\\\newline
Department of Mathematics and Statistics, UiT The Arctic University of Norway, Troms\o\  90-37,Norway \\\newline
    \textit{Email address: }{\href{mailto:omid.makhmali@uit.no}{\texttt{omid.makhmali@uit.no}}}
 }

\title[]
{A characterization of chains in dimension three} 
\date{\today}

\begin{abstract}
Given a 3-dimensional (para-)CR  structure, its family of  chains define a 3-dimensional path geometry. We provide necessary and sufficient conditions that determine whether a path geometry in dimension three arises from chains of a CR or para-CR 3-manifold.  We demonstrate how our characterization can be verified computationally for a given 3-dimensional path geometry and discuss a few examples.
\end{abstract}


\maketitle
  
\vspace{-.5 cm}

\setcounter{tocdepth}{2} 
\tableofcontents

\section{Introduction}
\label{sec:introduction}

Geometric structures on manifolds can give rise to distinguished set of curves  whose behaviour are of interest in a variety of problems. For example,  geodesics in (pseudo-)Riemannian geometry or in projective structures,  conformal geodesics in conformal geometry, null-geodesics in pseudo-Riemannian conformal geometry, and chains in CR structures or, more generally, in  contact parabolic geometries are some of the most well-known cases of such distinguished curves.

The main topic of interest in this article is  distinguished curves in 3-dimensional CR and para-CR structures. Such distinguished curves define 3-dimensional \emph{(generalized) path geometries}.  Recall that a (generalized) path geometry is locally defined  in terms of a set of paths on a manifold with the property that along each direction in (an open subset of) the tangent space of each point of the manifold, there passes a unique path in that family that is tangent to that direction. For instance, geodesics of an affine connection on a manifold define  a path geometry. Path geometries are a generalization of projective structures in which the curves may not be geodesics of any affine connection or satisfy any variational property. More generally, they may not be defined along every direction in each tangent space but may only be defined for an open set of directions at each point.

The  class of distinguished curves that we study are referred to as \emph{chains} and was originally defined in \cite{CM-CR} for CR structures.  On any CR manifold  chains are a canonical family of unparametrized curves  with the property that through each point and along each direction transverse to the contact distribution at that point, there passes a unique chain. As a result, chains define a canonical  generalized path geometry on any CR manifold.  Furthermore, for  para-CR structures,   defined as  contact manifold whose contact distribution has a splitting into two integrable Lagrangian distributions, it can be shown, in a similar manner, that chains can be defined  as well. As a result, any $(2n+1)$-dimensional  CR and para-CR geometry defines a canonical $(2n+1)$-dimensional path geometry via its chains.

The  objective of this paper is to give a set of invariant conditions for 3-dimensional path geometries which are satisfied if and only if the  3-dimensional path  geometry arises from chains of a CR or para-CR geometry. Our invariant conditions are computationally verifiable, i.e. if a 3-dimensional path geometry is presented either in terms of abstract structure equations or as a pair of second order ODEs, these conditions can be computed by manipulations that involve  roots of a  binary quartic and differentiation. The steps needed to verify the conditions are explained in detail within the proofs.

\subsection{Outline of the article and main results}
\label{sec:outline-article-main}

In  \ref{sec:path-geom-review} we give a review of CR, para-CR, and path geometry in dimension 3. The definition of a (generalized) path geometry on  an $n$-dimensional manifold $M$ is given in terms of a foliation of (an open set of) the $(2n-1)$-dimensional projectivized tangent bundle $\nu\colon\PP TM\to M$ by curves that intersect the fibers of $\PP TM\to M$ trivially. 
We recall some basic facts about path geometries in any dimension including their local realizability in terms of point equivalence classes of (systems of) 2nd order ODEs and a solution of their equivalence problem in the form of a Cartan connection. An important feature of 3-dimensional path geometries for us is that their fundamental invariants can be represented as a binary quadric $\bT,$ referred to as the torsion, and a binary quartic $\bC,$ referred to as the curvature.  As will be discussed in \ref{sec:path-geom-defin}, a generalized 3-dimensional path geometry on a 3-manifold $N$ is denoted by a triple $(Q,\scX,\scV)$ where $Q$ is 5-dimensional, $\scX,\scV\subset TQ$ intersect trivially and have rank 1 and 2, respectively,  their Lie bracket spans $T Q$ everywhere,  $\scV$ is integrable, and the local 3-dimensional leaf space of the foliation  induced by $\scV$ on $Q$ is an open set of $N.$ We define CR and para-CR structures on a contact manifold in terms of an integrable almost (para-)complex structure on the contact distribution that is compatible with the Levi bracket. In 3-dimensional contact manifolds the integrability of such a structure is automatic.  We recall the solution of the equivalence problem of (para-)CR structures. In Remark \ref{rmk:para-cr-2D-path} we will see that 3-dimensional para-CR structures are equivalent to (generalized) 2-dimensional path geometries.

In \ref{sec:an-overview-chains}, using the Cartan geometries associated to  (para-)CR structures, we give a Cartan geometric description of their chains, following \cite{CZ-CR}. We give a set of necessary conditions for a (generalized) 3-dimensional path geometry on $N,$ i.e. a triple $(Q,\scX,\scV),$ to arise from  the chains of a (para-)CR structure. In \ref{sec:corr-pair-2nd} we describe the 3-dimensional path geometry of chains in these two geometries in terms of the point equivalence class of pairs of second order ODEs. In particular, we discuss the chains of the flat CR structure on the 3-sphere  $\SSS^3\subset \CC^2$ and the flat para-CR geometry on the projectivized tangent bundle of $\PP^2,$ denoted as $\PP T\PP^2, $ and the induced (para-)K\"ahler-Einstein metrics on their respective space of chains, i.e. $\mathrm{SU}(2,1)\slash\mathrm{U}(1,1)$ and $\mathrm{SL}(3,\RR)\slash\mathrm{GL}(2,\RR)$.

In \ref{sec:path-geom-aris-1} and \ref{sec:3d-path-geometries} we prove our invariant characterization of chains. 
\begin{theorem}\label{thm:3D-path-CR-para-CR}
    A 3-dimensional path geometry $(Q,\scX,\scV)$ arises from chains of a 3-dimensional (para-)CR structure $(N,J,\scC)$ if and only if 
  \begin{enumerate}
  \item The binary quartic $\bC$ has two distinct  roots of multiplicity 2. 
  \item A canonical  2-form  $\rho\in\Omega^2(Q)$ of rank 2, defined via normalizing $\bC,$  is closed. 
  \item  The entries of the  binary quadric   $\bT$,  pulled-back by the bundle inclusion $\iota\colon\cG_{D}\hookrightarrow\cG$ defined via normalizing $\bC,$ have no dependency on the fibers of $Q\to N$ where  $N:=Q\slash\scV.$
  \end{enumerate}
\end{theorem}
It follows that in condition (1)  if one root is non-real complex of multiplicity two, then the other root has to be its complex conjugate, and the chains  correspond to a CR structure. In the case of two distinct real roots, the chains arise from a para-CR structure.
Furthermore,  Theorem \ref{thm:3D-path-CR-para-CR} implies that $Q$ is an open subset of  $\PP(TN\backslash \scC)$ where $\scC$ is a naturally induced
contact distribution on the 3-manifold $N$ equipped with a splitting. Theorem \ref{thm:3D-path-CR-para-CR} will be proved in two parts. The para-CR part is proved towards the end of \ref{sec:path-geom-aris-1}  and the CR part is proved at the end of \ref{sec:3d-path-geometries}.

In condition (2)    the canonical  2-form $\rho$ on $Q$ will be expressed in terms of the Cartan connection of the 3D path geometry on the reduced structure bundle $\cG_{D}\to Q,$ as given in Propositions \ref{prop:type-D-curv-chains} and \ref{prop:type-D-curv-CR-chains}. Such a closed 2-form is sometimes referred to as a \emph{quasi-symplectic} 2-form which is  the odd dimensional analogue of a symplectic 2-form, i.e. a closed 2-form of maximal rank. In Remark \ref{rmk:chains-3d-path-geometries-exclusive} we give an alternative  description of condition (3).

All three conditions in Theorem \ref{thm:3D-path-CR-para-CR} are in principle straightforward to check for any 3D path geometry expressed either as a pair of second order ODEs or in terms of abstract structure equations. 
Moreover, we show that chains are in fact a sub-class of a  larger class of 3-dimensional path geometries characterized by the first two conditions in Theorem \ref{thm:3D-path-CR-para-CR}. This is the content of  Theorem \ref{thm:2d-path-geometries-generalized-chains} in the case of para-CR structures.

\subsection{Conventions}
\label{sec:conventions}
Our consideration in this paper will be over smooth and real manifolds.  Throughout the article we always consider path geometry in the generalized sense and therefore will not use the term ``generalized'' when talking about path geometries. When defining the leaf space of a foliation we always restrict to sufficiently small open sets where the leaf space is smooth and Hausdorff.  Consequently, given an integrable distribution $\scD$ on a manifold $N,$ by abuse of notation, we denote the leaf space of its induced foliation by $N\slash\scD.$

We will use the summation convention over  repeated upper and lower indices. Given  elements  $v_1,\dots,v_k$ of a vector space, their span is denoted by $\langle v_1,\dots,v_k\rangle.$ When dealing with differential forms, the algebraic ideal generated by 1-forms $\alpha^1,\dots,\alpha^k$ is denoted as $\{\alpha^1,\dots,\alpha^k\}.$  Given a set of 1-forms $\alpha^0,\dots,\alpha^n,\beta^1,\dots,\beta^n,$ the corresponding dual frame is denoted as $\frac{\partial}{\partial\alpha^0},\cdots \frac{\partial}{\partial\beta^n}.$
On a principal bundle $\cG\to Q$ with respect to which the 1-forms $\alpha^0,\dots,\alpha^n,\beta^1,\dots,\beta^n$ give a basis for semi-basic 1-forms,   we define the \emph{coframe derivatives} of a function $f\colon\cG\to \RR$  as 
\begin{equation*}
  f_{;i}=\tfrac{\partial}{\partial\alpha^i}\im\exd f,\   f_{;ij}=\tfrac{\partial}{\partial\alpha^j}\im\exd f_{;i},\ f_{;\underline{a}}=\tfrac{\partial}{\partial\beta^a}\im\exd f,\  f_{;\underline{ab}}=\tfrac{\partial}{\partial\beta^b}\im\exd f_{;\underline a},\  f_{;\underline{a}i}=\tfrac{\partial}{\partial\alpha^i}\im\exd f_{;\underline a},\ f_{;i\underline{a}}=\tfrac{\partial}{\partial\beta^a}\im\exd f_{;i} 
\end{equation*}
and similarly for higher orders, where $0\leq i,j\leq n$ and $1\leq a,b\leq n.$ Note that in case we reduce the structure bundle of a geometric structure to a proper principal sub-bundle, by abuse of notation, we suppress the pull-back and use the same notation as above for the coframe derivatives on the reduced bundle. 

Lastly, given two distributions $\scD_1$ and $\scD_2,$ we denote by $[\scD_1,\scD_2]$ their \emph{derived  distribution},  i.e. the distribution whose sheaf of sections is  $\Gamma([\scD_1,\scD_2])=\Gamma(\scD_1)+\Gamma(\scD_2)+[\Gamma(\scD_1),\Gamma(\scD_2)].$  In the geometries we consider, i.e. 3D path geometries, the existence of such derived distribution is always guaranteed; see Definition \ref{def:generalized-path-geom}.   The sheaf of sections of  $\bigwedge^k(T^* M)$ is denoted by $\Omega^k(M).$
\section{A review of CR, para-CR, and path geometries}  
\label{sec:path-geom-review} 
In this section we  recall some of the well-known facts about path geometries, CR and para-CR geometries in dimension three, and a solution for their equivalence problem.

\subsection{Path geometries and systems of 2nd order ODEs}
\label{sec:path-geom-defin}

A path geometry  on an $(n+1)$-dimensional manifold $M$ is classically defined as a $2n$-parameter family of paths with the property that along each direction at any point of $M,$ there passes a unique path in that family. Consequently, the natural lift of the paths of a path geometry to the projectivized tangent bundle, $\PP TM,$ results in a foliation by curves which are transverse to the fibers $\PP T_xM.$ It is well-known that a path geometry on an $(n+1)$-dimensional manifold can be locally defined in terms of  a system of $n$ second order ODEs 
 \begin{equation}\label{systemODE}
   (z^i)''=F^i(t,z,z'),\quad t\in\R,\ \ z=(z^1,\dots,z^n),\quad 1\leq i\leq n,
\end{equation}
defined up to \emph{point transformations},  i.e.  
\[t\mapsto \tilde t=\tilde t(t,z^1,\ldots,z^{n}),\quad z^i\mapsto \tilde z^i=\tilde z^i(t,z^1,\ldots,z^{n}),\quad 1\leq i\leq n.\]
 Given a path geometry, its straightforward to see how it defines a  system of $n$ 2nd order ODEs: let $(z^0,\dots,z^n)$ be local coordinates on $V\subset M$. In a sufficiently small open set $U\subset \PP TM,$ the family of paths can be parametrized as $s\mapsto \gamma(s)=(z^0(s),\dots,z^n(s))$ for $s\in (a,b)\subset\RR.$ Each path is determined by the value of  $\gamma(s),$ and $\gamma'(s)$ at $s=s_0\in U,$ thus, taking another derivative, a system of $n+1$ 2nd order ODEs of the form $\gamma''= G(\gamma,\gamma')$ is obtained for a function   $G\colon \RR^{2n+2}\to\RR^{n+1}.$ If $U\subset \PP TM$ is sufficiently small, without loss of generality one can assume $\frac{\exd z^0}{\exd s}\neq 0$ in $U.$  Since the paths are given up to reparametrization, one is able to eliminate $s$ from the system $\gamma''=G(\gamma,\gamma').$ As a result, one arrives at the system of ODEs \eqref{systemODE} where $t:=z^0.$

Conversely, starting with an  equivalence class of system of $n$ 2nd order  ODEs under point transformations  \eqref{systemODE}, the system defines  a codimension $n$ submanifold $\cE\subset J^2(\RR,\RR^n)$ of the 2-jet space of $n$ functions of 1 variable.  Pulling-back  the canonical contact system on $J^2(\RR,\RR^n)$ to $\cE,$ one can identify $\cE$ with $J^1(\RR,\RR^n)$ which is additionally foliated by contact curves, i.e. the solution curves of the ODE system.  Locally, the fibration $J^1(\RR,\RR^n)\to J^0(\RR,\RR^n)\cong\RR^{n+1}$ can be identified as an open subset of the projectivized tangent bundle $\PP TM\to M$ for an $(n+1)$-dimensional manifold $M$ and the solution curves project to a $2n$-parameter family of paths on $M.$ However, for arbitrary ODE systems it may happen that  no path is tangent to some directions at some points of $M.$

In fact,  there are many instances of  path geometries that do not fit the classical description since the paths are only defined for an open set of directions. Thus, in order to study  path geometries one is led to work with a generalized notion of such structures as defined below, which is sometimes referred to as \emph{generalized path geometry}. However, in this article since path geometries for us are always defined in this generalized sense we will not use the term generalized.
\begin{definition}
\label{def:generalized-path-geom}
  An $(n+1)$-dimensional path geometry  is given by a triple $(Q,\scX,\scV)$   where $Q$ is a $(2n+1)$-dimensional manifold  equipped with a  pair of distributions $(\scX,\scV)$ of rank 1 and $n,$ respectively,  which intersect trivially and satisfy $[\scV,\scV]=\scV$ and $[\scX,\scV]=T Q$. The $(n+1)$-dimensional local leaf space of the foliation induced by $\scV,$ denoted as $M=Q/\scV,$ is said to be equipped with the path geometry $(Q,\scX,\scV)$. 
\end{definition}
 The rank $(n+1)$  distribution $\scC$ spanned by $\scX$ and $\scV$  induces a    \emph{multi-contact} structure on $Q,$ i.e., one can write  $\scC=\ker\{\alpha^1,\ldots,\alpha^n\}$ for some 1-forms $\alpha^0,\alpha^a,\beta^b,$ 
 such that  $\scX=\ker\{\alpha^a,\beta^b\}_{a,b=1}^n,$ $\scV=\ker\{\alpha^0,\ldots,\alpha^n\}$  and 
 \begin{equation}
   \label{eq:multi-contact}
   \exd \alpha^i\equiv \alpha^0\w\beta^i\ \ \mathrm{mod\ \ }\{\alpha^1,\ldots,\alpha^n\},
    \end{equation}
for all $1\leq i\leq n$.  Path geometry on surfaces corresponds to  $n=1$ in which case $\scC$ is a contact distribution on a 3-manifold $N.$ As a result, 2-dimensional path geometries have several unique features that do not extend to higher dimensional path geometries; see Remark \ref{rmk:para-cr-2D-path}.

One can easily check that Definition \ref{def:generalized-path-geom} is satisfied for a classical path geometry on an $(n+1)$-dimensional manifold by letting $Q=\PP TM$ and $\scV$ be the vertical tangent space of the fibration $\PP TM\to M$  and $\scX$ be the line field tangent to the natural lift of the path on $M$ to $\PP TM.$  In terms of the system of second order ODEs \eqref{systemODE} one has
\begin{equation}
  \label{eq:TotalDerivative}
  \scX=\spn\left\{\tfrac{\mathrm{D}}{\exd t}\right\}\quad \text{where}\quad \tfrac{\mathrm{D}}{\exd t}:=\partial_t+\sum_{i=1}^n p^i\partial_{z^i}+\sum_{i=1}^n F^i\partial_{p^i},
\end{equation}
and
\[
\scV=\spn\{\partial_{p^1},\ldots,\partial_{p^n}\},
\]
where $p^i$'s are fiber coordinates for an affine chart of  $\PP TM$. The vector field $\tfrac{\mathrm{D}}{\exd t}$ spanning $\scX$ is often called the total derivative vector field.

As was mentioned before not all geometries arising from Definition \ref{def:generalized-path-geom} are classical path geometries. However, restricting to  sufficiently small open sets   $U\subset Q$ in Definition \ref{def:generalized-path-geom},  $U$ can be realized as an open set of $\PP TM$ for the $(n+1)$-dimensional manifold $M$ which is the leaf space of $\scV.$ Consequently,  $\scX$ foliates $U\subset\PP TM$ by curves that are transverse to the fibers of $\PP TM\rightarrow M.$ We refer the reader to \cite[Section 2]{Bryant-ProjFlat} for more about path geometry on surfaces and \cite[Section 4.4.2,4.4.4]{CS-Parabolic}  in higher dimensions.

Two path geometries $(Q_i,\scX_i,\scV_i),$ $i=1,2$ are called equivalent if there exists a  diffeomorphism $f\colon Q_1\to Q_2$ such that $f_*(\scX_1)=\scX_2$ and $f_*(\scV_1)=\scV_2.$ 

To determine when two path geometries are locally equivalent, we provide a solution to the equivalence problem of path geometries using the notion of a Cartan geometry and Cartan connection as defined below.  
\begin{definition}
  Let $G$ be a Lie group and $P\subset G$ a Lie subgroup  with Lie algebras $\fg$ and $\fp\subset\fg,$ respectively.  A Cartan geometry of type $(G,P)$ on $Q,$ denoted as $(\cG\to Q,\psi),$ is a right principal $P$-bundle $\tau\colon\cG\to Q$ equipped with a Cartan connection $\psi\in\Omega^1(\cG,\fg),$ i.e. a $\fg$-valued 1-form on $\cG$ satisfying
  \begin{enumerate}
  \item $\psi$ is $P$-equivariant, i.e. $r_g^*\psi=\mathrm{Ad}_{g^{-1}}\circ\psi$ for all $g\in P.$
  \item  $\psi_z\colon T_z\cG\to \fg$ is  a linear isomorphism for all $z\in \cG.$
  \item $\psi$ maps fundamental vector fields to their generators, i.e. $\psi(\zeta_X)=X$ for any $X\in\fp$ where $\zeta_X(z):=\frac{\exd}{\exd t}\,\vline_{\,t=0}r_{exp(tX)}(z).$
  \end{enumerate} 
 The 2-form $\Psi\in\Omega^2(\cG,\fg)$ defined as
    \[\Psi(X,Y)=\exd\psi(X,Y)+[\psi(X),\psi(Y)]\quad \text{for\ \ }X,Y\in \Gamma(T\cG),\]
is called the Cartan curvature and is $P$-equivariant and semi-basic with respect to the fibration $\cG\to Q.$
\end{definition} 
The following solution of the equivalence problem for path geometries is due to Grossman and Fels  \cite{Grossman-Thesis,Fels}. 
\begin{theorem}\label{thm:path-geom-cartan-conn}
Every path geometry $(Q^{2n+1},\scX,\scV),$  as in Definition \ref{def:generalized-path-geom},   defines Cartan geometry  $(\cG\to Q,\psi)$ of type $(\mathrm{PSL}(n+2,\RR),P_{12})$ where $P_{12}\subset \mathrm{PSL}(n+2,\RR)$ is the parabolic subgroup preserving the flag of a line and a 2-plane in $\RR^{n+3}.$ Assume that  the distributions $\scX$ and $\scV$  are the projection of the distributions  $\ker\{\alpha^i,\beta^i\}_{i=1}^n$ and $\ker\{\alpha^i\}_{i=0}^n,$ respectively.   When $n\geq 2$, the Cartan connection and its curvature can be  expressed as
  \begin{equation}
  \label{eq:path-geom-cartan-conn-3D}  
  \psi=
  \def\arraystretch{1.3}
\begin{pmatrix}    
 -\psi^i_i-\psi^0_0 &\mu_0& \mu_j\\
\alpha^0 &\psi^0_0&\nu_j\\
\alpha^i&  \beta^i& \psi^i_j\\
\end{pmatrix}    
\quad \mathrm{and}\quad \Psi:=\exd\psi+\psi\w\psi=
  \def\arraystretch{1.3}
\begin{pmatrix}    
0 &M_0& M_j\\
0 &\Psi^0_0&V_j\\
0 &   B^i& \Psi^i_j\\
\end{pmatrix},    
    \end{equation}
where $1\leq i,j\leq n$ and 
\[B^i=T^i_j\alpha^0\w\alpha^j+\half T^i_{jk}\alpha^j\w\alpha^k,\qquad \mathrm{and}\quad \Psi^i_j=C^i_{jkl}\alpha^k\w\beta^l+\half T^i_{jkl}\alpha^k\w\alpha^l +T^i_{j0k}\alpha^0\w\alpha^k. \]
The fundamental invariants of a path geometry  are the \emph{torsion}, $\mathbf{T}=(T^i_j)_{1\leq i,j\leq n},$ and the \emph{curvature}, $\mathbf{C}=(C^i_{jkl})_{1\leq i,j,k,l\leq n}$, satisfying
\begin{equation}
  \label{eq:fels-invariants-path}
  T^i_i=0,\qquad C^i_{jkl}=C^i_{(jkl)},\qquad C^i_{ijk}=0.
\end{equation}
\end{theorem}
\begin{remark}\label{rmk:1-form-defined-on-G-or-Q}
  We point out that unlike in equation  \eqref{eq:multi-contact}, the 1-forms in Theorem \ref{thm:path-geom-cartan-conn} are defined on the principal bundle $\cG$ rather than the manifold $Q.$   If it is clear from the context that one needs to take a section $s\colon Q\to\cG$ and consider $s^*\psi,$ we will not make a distinction between 1-forms defined on $\cG$ or $Q.$
  \end{remark}

If $(t,z^i)$ and $(t,z^i,p^i)$ are   local jet coordinates on $\RR\times \RR^{n}=J^0(\RR,\RR^n)$ and $J^1(\RR,\RR^n),$ respectively, then, in terms of the system of ODEs \eqref{systemODE}, one obtains 
\begin{equation}
  \label{eq:fels-invariants-explicit}
  T^i_j=F^i_j-\textstyle{\frac{1}{2}\delta^i_jF^k_k},\qquad\qquad C^i_{jkl}=F^i_{jkl}-\textstyle{\frac{3}{4}}F^r_{r(jk}\delta^i_{l)}
\end{equation}
where $1\leq i,j,k,l\leq n$ and 
\begin{equation}
  \label{eq:fels-torsion-explicit-components}
  F^i_{j}=\textstyle{- \partial_{z^j}F^i+\frac{1}{2}\Dt (\partial_{p^j}F^i)-\frac{1}{4} \partial_{p^k}F^i\partial_{p^j}F^k,}\quad F^i_{jkl}=\partial_{p^j}\partial_{p^k}\partial_{p^l}F^i  ,\quad \Dt=\partial_t+p^i\partial_{z^i}+F^i\partial_{p^i}.
\end{equation}

Cartan's solution of the equivalence problem for path geometries on surfaces is treated in the next section since 2-dimensional path geometries coincide with 3-dimensional para-CR structures.

Here we are interested in path geometries on 3-dimensional manifolds, i.e.  $Q$ in Definition \ref{def:generalized-path-geom} is 5-dimensional, which correspond to point equivalence class of pairs of 2nd order ODEs. The following proposition  whose proof is straightforward and therefore  skipped, will be important for us.

\begin{proposition}\label{prop:3D-path-geom}
In three-dimensional path geometries the fundamental invariants $\bT$ and $\bC$ in Theorem \ref{thm:path-geom-cartan-conn}, as $\mathrm{GL}_2(\RR)$-modules,  can be represented as a  binary  quadric and a binary quartic, respectively, given by
\begin{equation}
  \label{eq:quadric-quartic}
  \begin{aligned}
 \bT&= s^*(A_0(\beta^1)^2+2A_1\beta^1\beta^2+A_2(\beta^2)^2)\otimes V\otimes X^{-2},\\
\bC&=s^*(W_0(\beta^1)^4+4W_1(\beta^1)^3(\beta^2)+6W_2(\beta^1)^2(\beta^2)^2+4W_3(\beta^1)(\beta^2)^3+W_4(\beta^2)^4)\otimes V\otimes X^{-1} \\
\end{aligned}
\end{equation}
where $s\colon Q\to\cG$ is a section, $\bT\in \Gamma(\mathrm{Sym}^2(\scV^*)\otimes\bigwedge^2\scV\otimes\scX^{-2}),$ $\bC\in \Gamma(\mathrm{Sym}^4(\scV^*)\otimes\bigwedge^2\scV\otimes\scX^{-1}),$ $X:=\frac{\partial}{\partial s^*\alpha^0}\in\Gamma(\scX),$  $V:=\frac{\partial}{\partial s^*\beta^1}\w\frac{\partial}{\partial s^*\beta^2}\in\Gamma(\bigwedge^2\scV),$ 
\[
\begin{gathered}
A_0=T^2_1,\quad A_1=T^2_2,\quad A_2=-T^1_2,\\
  W_0=C^2_{111},\quad W_1=C^2_{211},\quad W_2=C^2_{221},\quad W_3=C^2_{222},\quad W_4=-C^1_{222},
\end{gathered}
\]
and $(\tfrac{\partial}{\partial s^*\alpha^0},\tfrac{\partial}{\partial s^*\alpha^1},\tfrac{\partial}{\partial s^*\alpha^2},\tfrac{\partial}{\partial s^*\beta^1},\tfrac{\partial}{\partial s^*\beta^2})$ denote the vector fields dual to the coframe $s^*(\alpha^0,\alpha^1,\alpha^2,\beta^1,\beta^2)$ in the Cartan connection \eqref{eq:path-geom-cartan-conn-3D}.
Moreover, $Q$ is equipped with a degenerate conformal structure given by  $[s^*h]\subset\mathrm{Sym^2}(T^*Q)$ where
\begin{equation}
  \label{eq:degenerate-conformal-str}
  h:=\alpha^1\beta^2-\alpha^2\beta^1\in\Gamma(\mathrm{Sym}^2(T^*\cG)).
  \end{equation}
\end{proposition}
In what follows, we will need the Bianchi identities among the entries of the $\bC$ and $\bT$ given by
\begin{equation}
  \label{eq:W-A-curvature-torsion-Bianchies}
  \begin{aligned}
    \exd W_0 \equiv& 4 W_0 \psi^1_1 + 4 W_1 \psi^2_1,\quad     &&\exd W_1 \equiv W_0 \psi^1_2 + (3  \psi_1^1 +  \psi^2_2)W_1 + 3 W_2 \psi^2_1\\ 
    \exd W_2 \equiv& 2 W_1 \psi^1_2 + (2 \psi^1_1 + 2  \psi^2_2)W_2 + 2 W_3 \psi^2_1,\quad    && \exd W_3 \equiv 3 W_2 \psi^1_2 + (\psi^1_1 + 3 \psi^2_2)W_3 + W_4 \psi^2_1\\
    \exd W_4 \equiv& 4 W_3 \psi^1_2 + 4 W_4 \psi^2_2,\quad     &&\exd A_0 \equiv (4 \psi^0_0 + 3 \psi^1_1 + \psi^2_2)A_0 + 2 A_1 \psi^2_1\\
    \exd A_1 \equiv& A_0 \psi^1_2 + (4 \psi^0_0 + 2 \psi^1_1 + 2 \psi^2_2)A_1 + A_2 \psi^2_1,\quad 
    &&\exd A_2 \equiv 2 A_1 \psi^1_2 + (4 \psi^0_0 + \psi^1_1 + 3  \psi^2_2)A_2.
  \end{aligned}
\end{equation}
 modulo $\{\alpha^0,\alpha^1,\alpha^2,\beta^1,\beta^2\}.$ Using  Theorems \ref{thm:path-geom-cartan-conn} and \ref{thm:2D-path-geome}, it follows that two path geometries $(Q_i,\scX_i,\scV_i),$ $i=1,2$  are locally equivalent if and only if for their respective Cartan geometries $(\cG_i\to Q_i,\psi_i)$ there is a bundle diffeomorphism $f\colon\cG_1\to\cG_2$ such that $f^*\psi_2=\psi_1.$

Furthermore, if the fundamental invariants vanish, i.e. $\bC=\bT=0,$ then the Cartan curvature is zero and the Cartan connection satisfies the Maurer-Cartan equations of $\mathfrak{sl}(n+2,\RR)$ and hence the path geometry is locally equivalent to the canonical one on $\PP^{n+1}$ whose paths are projective lines.  If $\bC=0,$ then the path geometry defines a projective structure on the $(n+1)$-dimensional leaf space of $\scV,$ denoted by $M.$ If $\bT=0,$ then  the $2n$-dimensional leaf space of $\scX$, i.e. the  solution space of the corresponding system of ODEs, is endowed with a so-called $\beta$-integrable Segr\'e structure when $n\geq 2.$

Note that when $n=2$ and $\bT=0,$ then the induced $\beta$-integrable Segr\'e structure on the leaf space of $\scX,$ is  a 4-dimensional self-dual conformal structure of neutral signature which is defined by the conformal class of the bilinear form  $h$ in Proposition \ref{prop:3D-path-geom}. In this case $\bC$ gives the binary quartic representation of the self-dual Weyl curvature  of this  conformal structure.

\subsection{CR and para-CR 3-manifolds} 
\label{sec:cr-3-manifolds}
 
In this section we recall the definition of  CR and para-CR structures in dimension three and give a solution of their local equivalence problem. We will point out that para-CR structures in dimension 3 are equivalent to path geometries on surfaces and therefore can be described as the point equivalence class of  scalar 2nd order ODEs.

Recall that an almost complex structure on a  distribution $\scD\subset TN$ with even rank is an endomorphism $J_{-1}\colon \scD\to \scD$ such that $(J_{-1})^2=-\mathrm{Id}$ where $\mathrm{Id}\colon \scD\to \scD$ is the identity map. An almost para-complex structure on $\scD$ is defined by a linear map $J_{1}\colon \scD\to \scD$ such that $(J_{1})^2=\mathrm{Id}$ with the property that its $\pm 1$-eigenspaces have the same rank. As a result, we denote an almost (para-)complex structure by $J_\ve\colon \scD\to \scD$ such that $(J_\ve)^2=\ve\mathrm{Id}$ where $\ve$ is -1 and 1 for almost complex and para-complex structures, respectively. An almost (para-)complex structure is called \emph{integrable} if its $\pm\sqrt{\ve}$-eigenspaces are involutive as distributions in  $\KK_\ve\otimes \scD$ where $\KK_\ve:=\RR[\sqrt{\ve}],$ i.e. $\KK_{1}=\RR$ and $\KK_{-1}=\CC.$ Note that if $\scD$ has rank 2 then an almost (para-)complex structure on $\scD$ is always integrable.
\begin{definition}\label{def:cr-para-cr}
  A non-degenerate (para-)CR structure of hypersurface type on a $(2n+1)$-dimensional manifold $N$  is given by a contact distribution $\scC\subset TN$ with an integrable almost (para-)complex structure $J_\ve\colon\scC\to\scC$  such that $\mathfrak{L}(J_\ve X,J_\ve Y)=-\ve\mathfrak{L}(X,Y)$ where $\mathfrak{L}_x :(X,Y)\to [X,Y]_x\slash\scC_x\in T_xN\slash\scC_x$ for $X,Y\in\Gamma(TN)$ is the Levi bracket.  
\end{definition}
Identifying $TN\slash\scC$ at each point $x\in N$ with $\RR,$ the map $\mathfrak L_x$ corresponds to the imaginary part of a (para-)Hermitian inner product whose  signature is $(p,q), p+q=n$ when $\ve=-1$ and is $(n,n)$ when $\ve=1.$    The induced action of $J_\ve$ on $\KK_\ve\scC:=\KK_\ve\otimes\scC$ results in the splitting $\KK_\ve\scC=\scH_\ve\oplus\bar\scH_\ve,$ where $\scH_\ve$ and $\bar\scH_\ve$ are (para-)holomorphic and anti-(para-)holomorphic sub-bundles of $\KK_\ve\scC$ for the (para-)complex structure defined by $J_\ve,$ respectively. Note that as a result of the compatibility condition, in the para-CR case both $\scH$ and $\bar\scH$ are null with respect to the induced inner product on $\KK_1\scC\cong \scC.$ 
Furthermore, by the  compatibility condition in Definition \ref{def:cr-para-cr} in the para-CR case, it follows that  in the splitting  $\scC=\scH\oplus\bar\scH$ the integrable distributions $\scH$ and $\bar\scH$ are Lagrangian  with respect to the induced symplectic structure on $\scC.$ As a result, a para-CR structure is also referred to as an integrable contact Legendrian/Lagrangian structure (see \cite{CZ-CR,DMT} for more information.) As was mentioned above, the case of  contact 3-manifolds  is special since any almost (para-)complex structure on their contact distribution is automatically integrable.

Two (para-)CR structures, denoted as $(N,\scC,J_{\ve})$ and $(\tilde N,\tilde\scC,\tilde J_{\ve})$  are equivalent if there is a diffeomorphism $F\colon N\to \tilde N$ whose pushforward $F_*\colon TN \to T\tilde N$ and its extension to $\KK_\ve\otimes TN$ satisfy $F_*\scC=\tilde\scC$ and $F_*\scH_\ve=\tilde\scH_\ve$ and $F_*\bar\scH_\ve=\tilde{\bar \scH}_\ve$. Note that when $\ve=-1,$ the condition $F_*\bar\scH_\ve=\tilde{\bar \scH}_\ve$ is redundant.    Similarly, one can define  local equivalence of two (para-)CR structures at two points by finding such a diffeomorphism in sufficiently small neighborhoods of those points.

\begin{remark}\label{rmk:para-cr-2D-path}
Setting $n=1$ in Definition \ref{def:cr-para-cr},  one obtains that a 3-dimensional para-CR structure is given by a splitting $\scC=\scD_1\oplus \scD_2$ where $\scD_1$ and $\scD_2$ are rank one distributions. Furthermore, by Definitions  \ref{def:generalized-path-geom} it follows that a 2-dimensional (generalized) path geometry is defined by a contact 3-manifold whose contact distribution has a splitting into two line fields. Thus, the definition of 3-dimensional para-CR structures coincides with 2-dimensional path geometries. As a result, by our discussion in  \ref{sec:path-geom-defin} one can describe a 3-dimensional para-CR structure as the point equivalence class of a scalar 2nd order ODE, i.e. $z''=F(x,z,z')$.  In the case of  3-dimensional CR structures, when we want to give a  parametric expression we will consider induced CR structures on   hypersurfaces $N\subset \CC^2.$ In this case,  in some local coordinate system,    one can describe $N\subset \CC^2$ locally as  a graph  $\Im(w)=G(z,\bar z,\Re(w))$  where $(w,z)$ are local coordinates for $\CC^2.$ Such parametric descriptions of CR and para-CR structures will be used in \ref{sec:corr-pair-2nd}.
\end{remark}
Restricting to 3-dimensional CR structures, the following theorem of Cartan gives a solution of the equivalence problem in terms of a Cartan geometry. We refer the reader to \cite{Jacobowitz,Bryant-CR,CS-Parabolic} for more background on the geometry of CR structures.
\begin{theorem}[Cartan \cite{Cartan-CR2,Cartan-CR1}]\label{thm:CR-equiv-problem}
A 3-dimensional CR structure is a Cartan geometry  $(\cP\to N,\phi)$ of type $(\mathrm{SU}(2,1),P)$ where $P$ is the stabilizer of a  null line in $\CC^3.$  Expressing the Cartan connection as
   \begin{equation}
    \label{eq:psi-3D-CR}
    \phi=
    \begin{pmatrix}
      -\phi_0-\tfrac 13\ri\phi_1 & \ri\theta & \ri\theta_0\\
      \bar\omega & \tfrac 23\ri\phi_1 & -\ri\bar\theta\\
      \ri\omega^0& \omega & \phi_0-\tfrac 13\ri\phi_1
    \end{pmatrix},\quad \omega=\omega^1+\ri\omega^2,\quad \theta=\theta_1+\ri\theta_2.
  \end{equation}
 which  is $\mathfrak{su}(2,1)$-valued with respect to the Hermitian form $H$ in three variables 
\[H(Z)=Z_3\bar Z_1-Z_2\bar Z_2+Z_1\bar Z_3,\quad Z=(Z_1,Z_2,Z_3)\in\CC^3,\]
the Cartan curvature is given by
  \begin{equation}\label{eq:CartanCurv2Form}
    \Phi=\exd\phi+\phi\w\phi=
    \begin{pmatrix}
      0 & R\omega^0\w\omega & \ri(S\omega^0\w\omega+\bar S\omega^0\w\bar\omega) \\
      0 & 0 & \bar R\omega^0\w\bar\omega\\
      0 & 0 & 0
    \end{pmatrix},\quad R=R_1+\ri R_2,\quad S=S_1+\ri S_2.
  \end{equation}
  for two complex-valued functions $R$ and $S$ on $\cP.$ The structure function $R$ is defined up to a complex scale on $N$ and its vanishing implies  $\Phi=0$ and that the CR structure is locally equivalent to the hyperquadric $\Im(w)=\tfrac{1}{4}z\bar z$ in $\CC^2$ where $(z,w)$ are local coordinates on $\CC^2.$
\end{theorem}

In some literature the structure function $R$ is referred to as the fundamental invariant since it is defined up to a scale on $N$ and its vanishing implies flatness. Another choice of fundamental invariant is given in Remark \ref{rmk:absolute-invariants-cr-para-cr} which is invariantly defined on $N.$  Note that the 3-dimensional hyperquadric in $\CC^2,$ locally characterized by $R=0,$  is a graph of the compact quadric $z_0\bar z_2-z_1\bar z_1+z_2\bar z_0=0$ in $\CC\PP^2$ where $[z_0\colon z_1\colon z_2]$ are homogeneous coordinates using which $\CC^2$ is identified as $[1\colon z_1\colon z_2]$. It can be shown, e.g. see  \cite[Section 2.2]{Jacobowitz}, that it is CR equivalent to the 3-sphere $\SSS^3\subset\CC^2$ given by $w\bar w+z\bar z=1$.

Cartan’s solution of the equivalence problem for 3-dimensional para-CR structures  is as follows,   in which, using Remark \ref{rmk:para-cr-2D-path}, we denote a 3D para-CR structure by $(N,\scD_1,\scD_2)$ for two line fields $\scD_1,\scD_2\subset TN$ such that $\scD_1\oplus\scD_2$ is the contact distribution.
\begin{theorem}[Cartan \cite{Cartan-Proj}]\label{thm:2D-path-geome}
 Given a 3D para-CR structure  $(N,\scD_1,\scD_2)$ it defines a Cartan geometry $(\pi\colon\cP \to N,\phi)$ of type $(\mathrm{SL}(3,\RR),P_{12})$ where $P_{12}$ is the subgroup of upper diagonal matrices and
  \begin{equation}
  \label{eq:2D-path-geom-cartan-conn}  
  \phi=
  \def\arraystretch{1.3}
\begin{pmatrix}    
 -\phi^i_i & \theta_2 &\theta_0\\
 \omega^1 &\phi^1_1&\theta_1\\
\omega^0&  \omega^2& \phi^2_2\\
\end{pmatrix}.
\end{equation}
in which  $\scD_1=\pi_*\{\omega^0,\omega^2\}^\perp, \scD_2=\pi_*\{\omega^0,\omega^1\}^\perp.$ The Cartan curvature is given by
     \begin{equation}
            \label{eq:path-geom-curv}
\Phi=\exd\phi+\phi\w\phi=  \begin{pmatrix}
    0 & P_1 \omega^0 \w\omega^1 & P_2\omega^0 \w\omega^1+Q_2\omega^0\w\omega^2\\
    0 & 0 & Q_1\omega^0\w\omega^2\\
0 & 0 & 0
  \end{pmatrix}
\end{equation}
for some functions $P_1,P_2,Q_1,Q_2$ on $\cP.$ The structure functions $P_1$ and $Q_1$ are defined up to scales on $N$ whose vanishing  implies $\Phi=0$ and that the para-CR structure is locally equivalent to the flat model on the projectivized tangent bundle $\PP(T\PP^2)$ with its flat path geometry defined by the lift of projective lines on $\PP^2,$ locally described by the scalar ODE $z''=0.$
 \end{theorem}
 More explicitly, since by Remark \ref{rmk:para-cr-2D-path} 3D para-CR geometries  correspond to point equivalence class of scalar 2nd order ODEs, given a scalar ODE
 \[z''=F(t,z,z'),\]
 one obtains
 \[P_1=\tfrac{\mathrm{D^2}}{\exd t^2}F_{z'z'}-4\tfrac{\mathrm{D}}{\exd t}F_{zz'}-F_{z'}\tfrac{\mathrm{D}}{\exd t}F_{z'z'}+4F_{z'}F_{zz'}-3F_zF_{z'z'}+6F_{zz},\quad Q_1=F_{z'z'z'z'}.\]

\begin{remark}\label{rmk:absolute-invariants-cr-para-cr}
The fundamental invariant of a  3D CR structure   can be represented as $\bR=s^*(R\bar R(\omega^0)^4)\in\Gamma((\scC^\perp)^4)$ for any section $s\colon N\to \cP.$  Unlike $R,$ which is well-defined up to a scale, $\bR$ is a well-defined absolute invariant and its vanishing implies flatness of the CR structure.  Similarly, in para-CR 3-manifolds $(N,\scD_1,\scD_2),$ one obtains that  $\bR=s^*(P_1Q_1(\omega^0)^4)$ is an absolute invariant. However, in para-CR 3-manifolds the vanishing of $\bR$ is necessary but not sufficient for flatness. More precisely, if $\bR=0,$ then either $P_1=0$ or $Q_1=0.$ If $Q_1=0,$ then $Q_2=0$ and the 2-dimensional leaf space $M=N\slash\scD_2$ is equipped with a projective structure on $M$ for which  the absolute invariant  $\bP=s^*(P_1\omega^1+P_2\omega^0)\otimes(s^*\omega^0\w\omega^1)$ for any section $s\colon M\to\cP,$ constitute the only fundamental invariant. Similarly, if $P_1=0,$ then $P_2=0$ and the 2-dimensional leaf space $T=N\slash\scD_1$ is equipped with a projective structure for which $\bQ=s^*(Q_1\omega^2+Q_2\omega^0)\otimes (s^*\omega^0\w\omega^2),$ for any section $s\colon T\to\cP,$ is the only fundamental invariant. Thus, flatness of a 3D para-CR structure is equivalent to the vanishing of the \emph{relative invariants} $P_1$ and $Q_1.$ Alternatively, it is equivalent to the vanishing of the absolute invariant $\bR,$ on $N,$ and the subsequent  vanishing of  $\bP$ and $\bQ$ on $M$ and $T,$ respectively.   
\end{remark}

Lastly we recall the notion of orientability for  path geometries, following \cite[Section 2]{Bryant-ProjFlat}. In dimension two a path geometry  $(N,\scD_1,\scD_2)$ is said to be oriented if the leaves of $\scD_1$ and $\scD_2$ can be endowed with a continuous choice of orientation. This is equivalent to the existence of two non-vanishing vector fields $v_i\in\Gamma(\scD_i)$ on $N.$  A  1-form $\alpha$ is called $\scD_i$-positive if its pull-back to each leaf of $\scD_i$ is positive with respect to a  pre-assigned orientation. As a result, by the naturally induced co-orientation on the contact distribution due to the relation $\exd\omega^0\equiv\omega^1\w\omega^2$ mod $\{\omega^0\},$  having  $\scD_i$-positive 1-forms  $\omega^1$ and $\omega^2$ determines an orientation on $N.$ Since by definition the fibers of $N\to M=N\slash\scD_2$ and $N\to T=N\slash\scD_1$ are  connected, it is straightforward to show that an oriented path geometry determines an orientation on the leaf spaces of $\scD_1$ and $\scD_2.$  Similarly, one can define oriented path geometries $(Q,\scX,\scV)$ in higher dimensions by assigning a continuous choice of orientation to the leaves of $\scX$ and $\scV$ which imply the existence of a nonvanishing vector field spanning $\scX$ and a nonvanishing section of  $\bigwedge^n(\scV).$ As in the case of surfaces, via the multi-contact structure, one obtains an orientation on the leaf space of the foliation induced by $\scV.$

\section{Chains}
\label{sec:chain-construction}
In this section we review known facts about chains and derive necessary conditions for a 3-dimensional path geometry to be  defined by chains. We describe such path geometries in terms of their corresponding point equivalence class of pairs of second order ODEs.  In the last two subsections we prove Theorem \ref{thm:3D-path-CR-para-CR}.

\subsection{Chains in dimension 3}
\label{sec:an-overview-chains}

In the description of a CR structure in terms of a Cartan geometry $(\cP \to N,\phi),$ as given in Theorem \ref{thm:CR-equiv-problem}, consider the grading of $\mathfrak{su}(2,1)$ that corresponds to the parabolic subgroup $P,$ which  is a \emph{contact grading,} i.e.
\[\mathfrak{su}(2,1)=\fg_{-2}\oplus\fg_{-1}\oplus\fg_{0}\oplus\fg_1\oplus\fg_2,\]
where, in terms of \eqref{eq:psi-3D-CR}, $\omega^0\in\fg_{-2},$  $\omega,\bar\omega\in\fg_{-1},$ $\phi_1,\phi_2\in\fg_0,$ $\theta,\bar\theta\in\fg_1,$ and $\theta_0\in\fg_2.$  Now we can define chains as follows. 
\begin{definition}\label{def:chains-3d-cr}
Let $(N,\scC,J)$ be a 3-dimensional CR structure with Cartan geometric data $(\cP\to N,\phi).$   Taking  any element $X\in\fg_{-2},$ a chain on $N$ is a curve that is the projection of an integral curve of the vector field $\phi^{-1}(X)\in \Gamma(T\cP)$ via the projection $\tau\colon\cP\to N.$  
\end{definition}
In \cite{CZ-CR} the path geometry of  CR chains was studied using this Cartan geometric description, which we recall briefly in the case of dimension three. Firstly, one starts by describing the set of transverse directions to the contact distribution. Let $S\subset P$ be the stabilizer of the line $l\subset\mathfrak{su}(2,1)$  spanned by an element of $\fg_{-2}.$ It follows that  $\fs=\fg_{0}\oplus\fg_2\subset \fp=\fg_0\oplus\fg_1\oplus\fg_2$ where $\fs$ and $\fp$ are the  Lie algebras of $S$ and $P,$ respectively.  Moreover, it is straightforward to see that the  $P$-orbit of $l,$ denoted as $P\cdot l\subset\mathfrak{su}(2,1)\slash\fp,$ is the set of all lines in $\mathfrak{su}(2,1)\slash\fp$ that are not contained in $\fg^{-1}\slash\fp$ where $\fg^{-1}:=\fg_{-1}\oplus\fg_0\oplus\fg_1\oplus\fg_2$ and $\cP\times_P\fg^{-1}\slash\fp$ is the  contact distribution $\scC\subset TN.$ Thus, the bundle of directions that are transverse to the contact distribution, i.e.  $Q:=\PP(TN\backslash\scC),$ can be identified with $\cP\slash S:=\cP\times_P
P\slash S.$ Viewing the Cartan connection $\phi\in\Omega^1(\cP,\mathfrak{su}(2,1))$ as a connection on $Q,$ it gives 
\[TQ\cong\cP\times_S (\mathfrak{su}(2,1)\slash \fs).\] 
Using the above identification, it follows that the induced path geometry  $(Q,\scX,\scV)$ is give by $\scX=\nu_*(\phi^{-1}(\fg_{-2}))$ and $\scV=\nu_*(\phi^{-1}(\fg_{1}))$ where $\nu\colon\cP\to Q$ is the projection.

The relation between CR structures and the path geometry of their chains is studied via the so-called \emph{extension functor} in \cite{CZ-CR}. In the following proposition we state some necessary conditions for the path geometry of chains via an appropriate reduction, avoiding the notion of an extension functor.

\begin{proposition}\label{prop:chains-3d-cr-necessary-cond}
  Given a CR structure $(N,\scC,J)$ with corresponding Cartan geometry $(\cP\to N,\phi),$ let $(Q,\scX,\scV)$ be the path geometry of its chains with corresponding Cartan geometry $(\cG\to Q,\psi).$ Then, via the natural bundle map $\iota\colon\cP\to\cG,$ one has
  \begin{equation}
    \label{eq:CR-Chain-connection-pull-back}
    \def\arraystretch{1.3}
\iota^*\psi=
    \begin{pmatrix}
      -\phi_0 & -\theta_0 & -\theta_2 &  \theta_1\\
      \omega^0 & \phi_0 & \omega^2 & -\omega^1\\
      \omega^1 & \theta_1 & 0 & \phi_1 \\
      \omega^2 & \theta_2 & -\phi_1  & 0
    \end{pmatrix},
    \end{equation} 
  where $\phi$ is given as \eqref{eq:psi-3D-CR}. 
 Moreover, it follows that the fundamental invariants restricted to $\cP$ are given by 
  \begin{equation}
    \label{eq:TC-CR-to-3D}
    \bT_\cP=\left(-R_1(\theta_1)^2+R_1(\theta_2)^2+2R_2\theta_1\theta_2\right)\otimes V\otimes X^{-2},\quad \bC_{\cP}=3\left((\theta_1)^2+(\theta_2)^2\right)^2\otimes V\otimes X^{-1}
    \end{equation}
  where $\bT_{\cP}$ and $\bC_{\cP}$ are defined as \eqref{eq:quadric-quartic} using sections  $\iota\circ s\colon Q\to\cG$ and $s\colon Q\to\cP.$ Moreover the 2-form
\begin{equation}
  \label{eq:quasi-symp-CR-chains}
  \rho=s^*(\omega^1\w\theta_1+\omega^2\w\theta_2)\in\Omega^2(Q)
\end{equation}
  is well-defined and closed. The characteristic curves of $\rho$ coincide with chains of the CR structure. 
\end{proposition}
\begin{proof}
Using the fact that $\scX=\nu_*(\phi^{-1}(\fg_{-2}))$ and $\scV=\nu_*(\phi^{-1}(\fg_{1})),$ and the form of the Cartan connection $\phi$ in \eqref{eq:psi-3D-CR}, it is a matter of straightforward computation to show $\iota^*\psi$  and $\phi$ are related as in \eqref{eq:CR-Chain-connection-pull-back}. Using the Cartan curvature $\Phi$ in \eqref{eq:CartanCurv2Form} and the definition of $\bT$ and $\bC,$ it is a matter of straightforward  to obtain \eqref{eq:TC-CR-to-3D}. Lastly, it is elementary to show that $\rho$ is invariant under the action of the fibers $\cP\to Q$ and is closed using the structure equations \eqref{eq:psi-3D-CR}. The characteristic curves of $\rho$ are defined as the integral curves of the line field $l\subset TQ$ satisfying $l\im\rho=0.$ Using the expression \eqref{eq:quasi-symp-CR-chains}, it follows that $l=\langle\tfrac{\partial}{\partial t^*\omega^0}\rangle,$ for any section $t\colon Q\to \cP,$ which implies that    $l$ is the tangent direction to chains.
\end{proof}

Now we consider 3D para-CR geometries. By Remark \ref{rmk:para-cr-2D-path}, they are  defined in terms of a contact 3-manifold with the property that the contact distribution has a splitting. 
Similar to the case of CR structures, para-CR structures are equipped with a set of distinguished curves called chains.

To define chains in this case,  recall that the Lie algebra $\mathfrak{sl}(3,\RR)$ in which, by Theorem \ref{thm:2D-path-geome}, the Cartan connection  $\phi$ of a 3D para-CR geometry takes value has a contact grading
\[\mathfrak{sl}(3,\RR)=\fg_{-2}\oplus\fg_{-1}\oplus\fg_{0}\oplus\fg_{1}\oplus\fg_{2},\]
where $\fg_{-2},\fg_{2}$ have rank 1 and  $\fg_{-1},\fg_{1},\fg_0$ have rank  2. 

As a result the family of chains can be defined exactly as in Definition \ref{def:chains-3d-cr}. Similarly, it can be shown  that  chains of a 3D para-CR geometry $(N,\scD_1,\scD_2)$ define a 3-dimensional path geometry $(Q,\scX,\scV)$ where $N=Q\slash\scV.$ The paths of the induced path geometry on $N$ are defined for all directions that  are transverse to the contact distribution $\scC=\scD_1\oplus\scD_2\subset TN.$  In other words, the 5-manifold  $Q=\PP(TN\backslash\scC)$ is foliated by the natural lift of chains. Now we state a proposition analogous to Proposition \ref{prop:chains-3d-cr-necessary-cond} which gives  necessary conditions for a 3D path geometry to arise from chains of a para-CR geometry via an appropriate reduction. We skip the proof as it is as straightforward as that of Proposition \ref{prop:chains-3d-cr-necessary-cond}.
\begin{proposition}\label{prop:chains-2D-path-necessary-cond}
  Given a 3D para-CR geometry $(N,\scD_1,\scD_2)$ with corresponding Cartan geometry $(\cP\to N,\phi),$ let $(Q,\scX,\scV)$ be the path geometry of its chains with corresponding Cartan geometry $(\cG\to Q,\psi).$ Then, via the natural bundle map $\iota\colon\cP\to\cG,$ one has
  \begin{equation}
    \label{eq:2D-path-Chain-connection-pull-back}
    \def\arraystretch{1.3}
\iota^*\psi=
    \begin{pmatrix}
      -\phi^2_2-\half \phi^1_1 & \theta_0 & \half \theta_2 & \half \theta_1\\
      \omega^0 & \phi^2_2+\half\phi^1_1 & \half\omega^2 & -\half\omega^1\\
      \omega^1 & \theta_1 & \tfrac{3}{2}\phi^1_1 &  0\\
      \omega^2 & -\theta_2 & 0 & -\tfrac{3}{2}\phi^1_1
    \end{pmatrix},
    \end{equation}
  where $\phi$ is given as \eqref{eq:2D-path-geom-cartan-conn}. 
 Moreover, it follows that the fundamental invariants $\bT$ and $\bC$ restricted to $\cP$ are given by 
  \begin{equation}
    \label{eq:TC-3D-to-2D}
    \bT_\cP=\left(P_1(\theta_1)^2+Q_1(\theta_2)^2\right)\otimes V\otimes X^{-2},\quad \bC_{\cP}=6(\theta_1)^2(\theta_2)^2\otimes V\otimes X^{-1}
    \end{equation}
  where $\bT_{\cP}$ and $\bC_{\cP}$ are defined as \eqref{eq:quadric-quartic} using  sections  $\iota\circ s\colon Q\to\cG$ and  $s\colon Q\to\cP.$ Moreover the 2-form
\begin{equation}
  \label{eq:quasi-symp-chains}
  \rho=s^*(\omega^1\w\theta_2-\omega^2\w\theta_1)\in\Omega^2(Q)
\end{equation}
  is well-defined and closed. The characteristic curves of $\rho$ coincide with chains of the CR structure. 
\end{proposition}

By Proposition \ref{prop:chains-3d-cr-necessary-cond} and Proposition \ref{prop:chains-2D-path-necessary-cond}  if a 3D path geometry arises from chains of a para-CR or  CR structure, then there is a distinguished coframe in which the torsion and curvature can be put in the form \eqref{eq:TC-3D-to-2D} and \eqref{eq:TC-CR-to-3D}, respectively.
\begin{remark}\label{rmk:R-T-discriminant}
The torsion $\bT_\cP,$ as expressed  in Proposition \ref{prop:chains-3d-cr-necessary-cond} and \ref{prop:chains-2D-path-necessary-cond}, can be related to the absolute invariant $\bR,$ as defined in Remark \ref{rmk:absolute-invariants-cr-para-cr}, via $\bR=s^*(\Delta_{\bT_\cP}(\omega^0)^4)$ for any section $s\colon N\to \cP,$ where $\Delta_{\bT_\cP}$ is the discriminant of the quadratic polynomial given by $\bT_\cP.$
\end{remark}
\begin{remark}
In the language of \cite{MS-cone}, the 2-form  $\rho$ in Proposition \ref{prop:chains-3d-cr-necessary-cond} and Proposition \ref{prop:chains-2D-path-necessary-cond}  induces a  \emph{compatible quasi-symplectic structure} on $Q$, i.e. $\rho\w\rho\neq 0$ and $\exd\rho=0$ with the property that the fibers of $Q\to N$ are isotropic and the tangent directions to the paths, $\langle\tfrac{\partial}{\partial\omega^0}\rangle,$ coincide with its characteristic direction. 
 
\end{remark}
 
\subsection{Corresponding pairs of 2nd order ODEs}
\label{sec:corr-pair-2nd}

In this section we use the fact that chains are the characteristics of the 2-form $\rho$ in Proposition \ref{prop:chains-3d-cr-necessary-cond} and Proposition \ref{prop:chains-2D-path-necessary-cond} to associate a pair of second order ODEs to their corresponding 3D path geometry. Using the derived ODEs, we give an explicit description of the induced geometric structure on the space of chains for flat (para-)CR structures. 

\subsubsection{Derivation of ODEs}
\label{sec:chains-pairs-odes}

By Remark \ref{rmk:para-cr-2D-path}, a 3D para-CR geometry is locally given by the point equivalence class of a scalar ODE $y''=F(z,y,y').$ As was discussed in \ref{sec:path-geom-defin}, the hypersurface $\cE\subset J^2(\RR,\RR)$ defined by the ODE, can be identified with $J^1(\RR,\RR).$ The contact distribution of $J^1(\RR,\RR)$ thus inherits a splitting $\scC=\scD_1\oplus\scD_2$ where $\scD_2$ is the vertical tangent bundle to $J^1(\RR,\RR)\to J^0(\RR,\RR)$ and $\scD_1$ is the tangent direction to the solution curves of the ODE. Let $(x,y)$ and $(x,y,p)$ be local coordinates for $J^0(\RR,\RR)$ and $J^1(\RR,\RR),$ respectively. A choice of adapted coframe on $N$ for the path geometry is given by 
\[s^*\omega^0=\exd y-p\exd x,\quad s^*\omega^1=\exd x,\quad s^*\omega^2=\exd p-F(x,y,p)\exd x,\]
where $s\colon J^1(\RR,\RR)\to \cP$ is a section of the principal bundle of the 3D para-CR geometry. It is a matter of standard calculation, e.g. see \cite{Gardner-Book}, to find all the entries of $s^*\phi.$ In particular, one obtains
\begin{equation}
  \label{eq:theta12-lift-2D-path-parametric}
    \begin{aligned}    
      s^*\theta_1=&\tfrac 16F_{ppp}(\exd y-p\exd x),\\
      s^*\theta_2=&\left(\tfrac 16 \tfrac{\rD}{\exd x}(pF_{pp})-\tfrac 23\tfrac{\rD}{\exd x}F_p +\tfrac 23 F_{xp}+F_y\right)\exd x+(\tfrac 23 F_{yp}-\tfrac 16\tfrac{\rD}{\exd x}F_{pp})\exd y+\tfrac 12 F_{pp}\exd p.
    \end{aligned}
\end{equation}
where $\tfrac{\rD}{\exd x}=\partial_x+p\partial_y+F\partial_p$ is the total derivative  \eqref{eq:TotalDerivative}. To introduce a coordinate system on $Q:=\PP(TN\backslash\scC),$ which we identified with $\cP\slash S$ in \ref{sec:an-overview-chains}, we shall first parametrize the fibers of $\cP\to N$ which are the upper triangular matrices $P_{12}\subset\mathrm{SL}(3,\RR).$ Writing $P_{12}=P_0\ltimes P_+$ where $P_0$ is the reductive subgroup, referred to as the structure group, and $P_+$ is the nilpotent normal subgroup, one obtains
\[P_0=\left\{A\in\mathrm{SL}(3,\RR)\,\vline\, A=
    \begin{pmatrix}
      \tfrac{1}{a_1a_2} & 0 & 0\\
      0 & a_1 & 0\\
      0 & 0 & a_2
    \end{pmatrix}
\right\},\ P_+=\left\{B\in \mathrm{SL}(3,\RR)\,\vline\, B=
  \begin{pmatrix}
    1 & b_2 & \half b_1b_2+b_0\\
    0 & 1 & b_1\\
    0 & 0 & 1
  \end{pmatrix}
\right\}\]
Thus, the variables $(x,y,p,a_1,a_2,b_0,b_1,b_2)$ give a local coordinate system for $\cP.$ By the discription of $S$ and that $Q\cong\cP\slash S,$ it follows that $Q$ can be identified with the slice $a_1=a_2=1$ and $b_0=0.$

Now we lift the adapted coframe on $N$ to $Q.$ In order to do so, we recall the transformation of the Cartan connection along the fibers of $\cP\to N$ to be
\begin{equation}
  \label{eq:gauge-transformation-2D-path}
\phi(z)\to \phi(r_gz)=g^{-1}\phi(z)g+g^{-1}\exd g,
\end{equation}
where $z\in\cP$ and $g\in P_{12}.$ Since $Q$  is identified with $a_1=a_2=1,b_0=0,$ then $g$ can be written as $g= B$ where  $B\in P_+$ as parametrized above.
Using the section $t\colon Q\to \cP$ given by $a_1=a_2=0,b_0=0,$ and the expressions of $s^*\theta_1$ and $s^*\theta_2,$  an adapted  coframe on $Q$ is given by
\[
  \begin{aligned}
    s^*\omega^0 &\to t^*\omega^0=\exd y-p\exd x,\\
    s^*\omega^1 & \to t^*\omega^1=\exd x-b_1(\exd y-p\exd x),\\
    s^*\omega^2 & \to t^*\omega^2=\exd p -F(x,y,p)\exd x-b_2(\exd y-p\exd x),\\
    s^*\theta_1 & \to t^*\theta_1=\exd b_1+z_{11}\exd x+z_{12}\exd y+z_{13}\exd p, \\
    s^*\theta_2 & \to t^*\theta_2=\exd b_2+z_{21}\exd x+z_{22}\exd y +z_{23}\exd p,
  \end{aligned}\]
for functions $z_{ij}$ on $Q$  determined via \eqref{eq:gauge-transformation-2D-path}. 

Now we can explicitly find the quasi-symplectic 2-form $\rho$ on $Q$ to be
\begin{equation}
  \label{eq:rho-2D-path}
  \begin{aligned}
    \rho&=t^*(\omega^1\w\theta_2-\omega^2\w\theta_1)\\
    &= -\exd p \w \exd b_1-\tfrac 16F_{ppp}\exd p \w \exd y+(b_2 b_1+\half F_{pp}+b_1F_{p}-\tfrac 16 pF_{ppp} ) \exd x\w\exd p+(b_1 p+1) \exd x \w\exd b_2\\
    &+ (p b_2+F) \exd x \w \exd b_1+(-\tfrac 16 F_{xpp}+\tfrac 23 F_{yp}+b_1 F_y-\tfrac 16 p F_{ypp}) \exd x \w \exd y-b_1 \exd y \w \exd b_2-b_2 \exd y \w \exd b_1
  \end{aligned}
\end{equation}
It is straightforward to find a characteristic vector field for $\rho,$ i.e. $X\in\Gamma(TQ)$ such that  $X\im\rho=0.$ It follows that $X=\langle\partial_x+\tfrac{pb_1+1}{b_1}\partial_y+\tfrac{Fb_1-b_2}{b_1}\partial_p+B_1\partial_{b_1}+B_2\partial_{b_2}\rangle$ for two functions $B_1$ and $B_2.$ In order to get a pair of ODEs, we use coordinate manipulation to make $X$ look like a total derivative  \eqref{eq:TotalDerivative}. We take $x$ to be the independent variable of the pair of ODEs and introduce the following new variables
\begin{equation}
  \label{eq:coordinate-change-2Dpath-chains}
  Y=\tfrac{pb_1+1}{b_1},\quad P=\tfrac{Fb_1-b_2}{b_1}.
  \end{equation}
Expressing $\rho$ in the new coordinate system $(x,y,p,Y,P)$ to find a characteristic vector field $X,$ one obtains $X=\langle\partial_x+Y\partial_y+P\partial_p+G_1\partial_Y+G_2\partial_P\rangle$ for two functions $G_1,G_2$ on $Q.$ Replacing $Y$ with $y'$ and $P$ with $p'$ in $G_1$ and $G_2,$ the  pair of second order ODEs that corresponds to the characteristic curves of $\rho$ is given by 
\begin{equation}
  \label{eq:chain-pair-ODEs}
      \begin{aligned}
      y'' =& F + F_p\Delta+\half F_{pp}\Delta^2 +\tfrac{1}{6} F_{ppp}\Delta^3\\
      p''=&\tfrac{2}{\Delta}(p'-F)^2+F_p(3p' - 2F ) + F_x + pF_y + \left(F_{pp}(p' - F ) + 2F_y\right)\Delta\\
      & +\tfrac{1}{6}\left(F_{ppp} (p' - 2F) - F_{xpp} + 4F_{yp} - pF_{ypp})\right)\Delta^2
    \end{aligned}
  \end{equation}
  where   $\Delta=y'-p.$ In particular, the pair of second order ODEs for the chains of the flat 3D para-CR geometry, which corresponds to $F=0$ in \eqref{eq:chain-pair-ODEs}, is given by
  \begin{equation}
    \label{eq:chains-flat-2D-path}    
    y''=0,\quad p''=\frac{2(p')^2}{p-y'}.
      \end{equation}
  The pair of ODEs  \eqref{eq:chain-pair-ODEs}  was obtained in \cite{BW-chains} via a different approach. It is known that chains of a 3D para-CR geometry can be defined as the projection of null geodesics of the corresponding \emph{Fefferman conformal structure} which has signature (2,2). In \cite{BW-chains}  the authors  used the lifted coframe to derive the equation of null geodesics for the correponding Fefferman conformal structure in order to parametrize the chains.     The reason that the pair of ODEs obtained from either approaches agree is due to that fact that a (para-)CR structure is, in particular, a symmetry reduction of its corresponding Fefferman conformal structure by a null conformal Killing field. As is shown in \cite{MS-cone}, in a symmetry reduction the null geodesics of the conformal structure define a \emph{variational orthopath geometry} on the leaf space of the infinitesimal symmetry. The  paths of a variational orthopath structure are characteristic curves of a quasi-symplectic structure. In our case the leaf space is a (para-)CR manifold and the paths of the orthopath geometry are the chains since they are the projection of the null geodesics of a Fefferman conformal structure. Thus, to express the system of ODE for chains one can either express null geodesics and project them or directly compute the characteristic curves of the orthopath geometry of chains.

\begin{remark}\label{rmk:sl3-submaximal-3rd-order-ODE-o22}
 Instead of $x,$ one can take $y$ or $p$ as the independent variable when deriving  a pair of ODEs that corresponds to chains. This would result in a pair of ODEs that is point equivalent to \eqref{eq:chain-pair-ODEs}.   
Note that the pair of ODEs \eqref{eq:chains-flat-2D-path}  has 8-dimensional algebra of infinitesimal symmetries which is isomorphic to $\mathfrak{sl}(3,\RR)$. It can be  solved explicitly. Using the second ODE, $y'$ can be solved algebraically. Substituting  in the second ODE, one obtains the 3rd order ODE
\begin{equation}
  \label{eq:submax-3rd-order-ODE-1}
  p'''=\frac{3(p'')^2}{2p'},
  \end{equation} 
which is one of the two  submaximal 3rd order ODEs under point transformations \cite[Section 4.2]{Godlinski}. This third order ODE is a re-expression of the vanishing of the Schwarzian derivative of $p(x).$ Its algebra of infinitesimal symmetries is  $\mathfrak{o}(2,2)$  and has vanishing \emph{W\"unschmann invariant} and \emph{Cartan invariant}  which implies that it induces a non-flat \emph{Einstein-Weyl structure} on its solution space (see \cite{Tod} for a discussion on Einstein-Weyl geometry and \cite{Wojtek} for the variational properties of \eqref{eq:submax-3rd-order-ODE-1}.) Since the Weyl 1-form is closed, such Einstein-Weyl structures locally define 3-dimensional Lorentzian metrics up to homothety. In fact, they are the homothety class of the Lorentzian metric of negative sectional curvature on  $\mathrm{SO}(2,2)\slash\mathrm{SO}(2,1)\cong\HH^{2,1}.$ Lastly, we point out that the relation above between a pair of second order ODEs arising as chains and a scalar third order ODE remains valid if the initial 3D para-CR geometry is defined by an ODE of the form $y''=F(x,y')$ i.e. if the 3D para-CR geometry has an infinitesimal symmetry. Note that any infinitesimal symmetry of a CR structure is almost everywhere transverse to the contact distribution. More precisely, if $F=F(x,y')$ then using the second ODE  in \eqref{eq:chain-pair-ODEs} one can find $y'$ as a function of  $x,p,p',p''.$ Subsequently, replacing in the first ODE in \eqref{eq:chain-pair-ODEs} one obtains a third order ODE in $p(x).$ This relation between a pair of second order ODEs and a scalar third order ODE, also appears in \cite[Section 6]{DW} using  constructions that are in general different from chains unless the 3D para-CR geometry is flat.   Unlike the construction in \cite[Section 6]{DW}, we do not expect that 3rd order ODEs obtained in this fashion from pairs of ODEs defined by the chains of a non-flat scalar ODE $y''=F(x,y')$ to correspond to an Einstein-Weyl structure.
\end{remark}

Now we would like to derive the pair of 2nd order ODEs corresponding to chains of  the 3D CR structure induced on a hypersurfaces $N\subset \CC^2$. Following Remark \ref{rmk:para-cr-2D-path}, we describe $N\subset \CC^2$  as a graph  $\Im(w)=G(z,\bar z,\Re(w))$  where $(w,z)$ are local coordinates for $\CC^2.$ Changing to real coordinates, we write $z=x+\ri y$ and $w=p+\ri q.$ Then the graph is given by $q=F(x,y,p).$  One  obtains an adapted coframe on $N$ by noticing that the holomorphic contact directions are given by the Lewy operator $l=\partial_z+B(x,y,p)\partial_w$ and should satisfy $l\im\exd(q-F(x,y,p))=0.$ Thus, by restricting to $N,$ one has $B|_{N}=\frac{\ri F_y-F_x}{\ri+F_p}.$ As a result, since $\scH=\langle l\rangle,$ it follows that an adapted coframe on $N,$ corresponding to a section $s\colon N\to \cP,$ is given by
\begin{equation}
  \label{eq:CR-adapted-cof-on-N}
  s^*\omega^0=\tfrac{1}{C}\left(\exd p + 2\tfrac{F_xF_p-F_y}{F_p^2+1}\exd x+ 2\tfrac{F_yF_p+F_x}{F_p^2+1}\exd y \right),\quad s^*\omega^1=\exd x,\quad s^*\omega^2=-\exd y,
  \end{equation}
where
\[C=\tfrac{(2F_xF_p-F_p^2F_y-3F_y)F_{xp}+(2F_yF_p+F_p^2F_x+3F_x)F_{yp}-2(F_x^2+F_y^2)F_{pp}-(F_p^2+1)(F_{xx}+F_{yy})}{(F_p^2+1)^2}.\]
We refer the reader to \cite[Chapter 6]{Jacobowitz} for more detail. Similar to the case of 3D para-CR geometries, one can find the explicit form of $s^*\theta_1$ and $s^*\theta_2,$ however, unlike \eqref{eq:theta12-lift-2D-path-parametric},  the expressions are extremely long and will not be provided here. Repeating what we did for 3D para-CR geometries, one needs to lift the adapted coframe \eqref{eq:CR-adapted-cof-on-N} on $N$ to $Q.$ As a result, knowing $P=P_0\ltimes P_+,$ we restrict to $P_0=\mathrm{Id}$ and we parametrize the nilpotent part, $P_+\subset P$ as
\[P_+=\left\{B\in \mathrm{SU}(2,1)\,\vline\, B=
    \begin{pmatrix}
      1 & \ri(b_1+\ri b_2) & \half(b_1^2+b_2^2)+\ri b_0\\
      0 & 1 & -\ri(b_1-\ri b_2)\\
      0 & 0 & 1
    \end{pmatrix}
\right\}.\]
Since one has $Q\cong \cP\slash S,$ it follows that $Q$ can be identified by setting $P_0=\mathrm{Id}$ and $b_0=0.$ Thus,  $(x,y,p,b_1,b_2)$ give a local coordinate system on $Q.$
Lifting $s^*\omega^i$'s, $s^*\theta_1$ and $s^*\theta_2$ to $Q$ via the prescribed section $t\colon Q\to \cP$ can be carried out identically as in the para-CR case. This allows one to compute the quasi-symplectic 2-form $\rho=t^*(\omega^1\w\theta_1+\omega^2\w\theta_2)$ and find its characteristic direction $\scX=\langle X\rangle$ by solving $X\im\rho=0.$  To put $X$ in the form of a total derivative one needs to carry out a change of variable analogous to \eqref{eq:coordinate-change-2Dpath-chains} although the expressions involved are much longer. This will consequently give the desired pair of ODEs.

Since the resulting pair of ODEs for chains of a general CR manifold, locally given as $q=F(x,y,p),$ cannot be  written here for an arbitrary $F$ due to its length, we consider two simple cases. To describe the flat  CR structure on the 3-sphere $\SSS^3\subset\CC^2$ we use the fact that it is equivalent to the  hyperquadric  $\QQ^3\subset\CC^2$; see the discussion following Theorem \ref{thm:CR-equiv-problem}. Thus, putting $F(x,y,p)=\tfrac{1}{4}(x^2+y^2),$ the resulting pair of ODEs for its chains is found to be 
\begin{equation}
  \label{eq:chains-CR-pairs-2nd-order-ODE-F-x2}
  \begin{aligned}
y''=\frac{((y')^2+1)^2}{y' x+p'-y},\quad   p''= \frac{((y')^2+1) (p' y'-y y'-x)}{y' x+p'-y}
    \end{aligned}
\end{equation}
where  $x$ is taken as the independent variable. This pair of ODEs is torsion-free and has 8-dimensional algebra of infinitesimal symmetries isomorphic to $\mathfrak{su}(2,1)$.  Using the first ODE, $p'$ can be solved algebraically. Substituting  in the second ODE, one obtains the 3rd order ODE
\begin{equation}
  \label{eq:submax-3rd-order-ODE-2} 
  y'''=\frac{3y'(y'')^2}{1+(y')^2},
  \end{equation}
  which,  together with \eqref{eq:submax-3rd-order-ODE-1}, are the only two  submaximal 3rd order ODEs under point transformations \cite[Section 4.2]{Godlinski}. Its algebra of infinitesimal symmetries is  $\mathfrak{o}(3,1)$  and has vanishing \emph{W\"unschmann invariant} and \emph{Cartan invariant}, and, thus, induces a non-flat Einstein-Weyl structure on its solution space. Since its Weyl 1-form is closed, it is locally determined by a Lorentzian metric up to homothety. In fact, it is the homothety class of the Lorentzian metric of positive sectional curvature on $\mathrm{SO}(3,1)\slash\mathrm{SO}(2,1)\cong\SSS^{2,1}.$ 

  More generally, it turns out that if the graph of a CR structure  can be put in the form $\Im(w)=G(z,\bar z),$ i.e. the CR structure has an infinitesimal symmetry, then the same construction mentioned in Remark \ref{rmk:sl3-submaximal-3rd-order-ODE-o22} goes through. In other words, one can associate a third order ODE to the pair of second order ODEs defined by the chains of  CR structures with an infinitesimal symmetry. We do not have an invariant description or characterization of this construction. Moreover, as in the case of CR structures, any infinitesimal symmetry of a para-CR structure is almost everywhere transverse to the contact distribution.

The pair of ODEs  for the chains of a CR structure that is not locally equivalent to the flat model  becomes  cumbersome to write. As an example, taking $F(x,y,p)=\tfrac{1}{6}y^3,$ the pair of ODEs for its chains are given by
\[
  \begin{aligned}
    y''&=\tfrac{16 (y')^4 y^6+22 (y')^2 y^6+12 p' (y')^2 y^4-2 (p')^2 (y')^2 y^2+5 y^6+15 p' y^4-5 (p')^2 y^2+(p')^3}{16y^5 (p'-y^2)},\\
    p''&=\tfrac{(8 (y')^2 y^4+15 y^4+10 p' y^2-(p')^2) y'}{8y^3}.
  \end{aligned}
\]
Following our discussion above  on chains of CR structures with an infinitesimal symmetry and 3rd order ODEs,  in the above pair of ODEs  $p'$ can be solved algebraically from the first ODE. Replacing the  expression for $p'$  into the second ODE gives a 3rd order ODE in $y(x).$ It would be interesting to give a geometrically invariant description of this correspondence.

\begin{remark}\label{rmk:complex-point-trans}
The Lie algebras $\mathfrak{sl}(3,\RR)$ and $\mathfrak{su}(2,1)$ are real forms of $\mathfrak{sl}(3,\CC).$ Thus, over $\CC$ flat para-CR and CR structures are locally equivalent. As a result, the path geometry of their chains over $\CC$ are equivalent. In other words, the pairs of ODEs \eqref{eq:chains-CR-pairs-2nd-order-ODE-F-x2} and \eqref{eq:chains-flat-2D-path} define the same equivalence class under complex-valued point transformations. To express the complex-valued point transformation that relates these two pairs of ODEs, we  denote the $(x,y,p)$ variables in pair of ODEs \eqref{eq:chains-flat-2D-path} with tildes. Then the point transformation on $J^0(\CC,\CC^2)$ that sends the pair of ODEs \eqref{eq:chains-CR-pairs-2nd-order-ODE-F-x2} to \eqref{eq:chains-flat-2D-path} is
\begin{equation}
  \label{eq:complex-point-trans}
  (\tilde x,\tilde y,\tilde p)=(x-\ri y,\half(x^2+y^2)+\ri p,x+\ri y).
  \end{equation}
Moreover, using this point transformation  the submaximal 3rd order ODEs \eqref{eq:submax-3rd-order-ODE-2} and \eqref{eq:submax-3rd-order-ODE-1} are also equivalent over $\CC.$  The point transformation that relates the submaximal 3rd order ODEs above was  found in \cite[Remark 6.3]{KT-ODE} in the larger context of submaximal systems of 3rd order ODEs.
  
\end{remark}

\subsubsection{Space of chains}
\label{sec:space-chains}

Now we would like to briefly  describe the induced geometry on the space of chains for the flat para-CR geometry on $\PP T \PP^2,$ or equivalently, the flat path geometry on $\PP^2.$ In \ref{sec:an-overview-chains} chains were defined as the integral curves of $\phi^{-1}(X)$ for any $X\in \fg_{-2}$ and $Q=\cP\slash S.$  Since for the flat  path geometry on $\PP^2$ one has $\cP=\mathrm{SL}(3,\RR)$, the space of its chains is $Z:=Q\slash\mathrm{exp}(\fg_{-2})\cong \mathrm{SL}(3,\RR)\slash\mathrm{GL}(2,\RR)$ which can be identified with $\PP^2\times(\PP^2)^*\backslash N,$ where $N$ is the 3-manifold defined by  pairs of  incident points $(a,b)\in\PP^2\times(\PP^2)^*$. The reader may see \cite[Section 4.1]{BW-chains} for more detail.

  When the path geometry is flat  there is a  naturally induced self-dual \emph{para-K\"ahler-Einstein} metric on $Z$ which has been thoroughly studied and is sometimes referred to as the para-Fubini-Study metric or the dancing metric; see \cite{DW,PKE}.  Moreover, as can be seen in Proposition \ref{prop:chains-3d-cr-necessary-cond}, the torsion of 3D path geometries arising from chains of para-CR geometries is zero if and only if the para-CR geometry is flat. Using the twistor correspondence between torsion-free 3D path geometries and self-dual conformal structures \cite{Grossman-Thesis}, it follows that  $Z$ carries  a canonical  conformal structure of neutral signature if and only if the underlying para-CR geometry is flat.

In order to express the induced para-K\"ahler-Einstein metric on $Z,$ we would like to use the pair of ODEs \eqref{eq:chains-flat-2D-path} in the following way. Let $(x,y,p)$ and $(x,y,p,Y,P)$ be a coordinate system on $J^0(\RR,\RR^2)$ and $J^1(\RR,\RR^2)$ with $x$ as the independent variable. Following \cite{DFK,KM-Cayley}, the space of solutions of the point equivalence class of a pair of 2nd order ODEs  can be identified with  a hypersurface $x=\mathrm{const}.$ Given a pair of torsion-free 2nd order ODEs  $y''=F(x,y,p,y',p')$ and $p''=G(x,y,p,y',p'),$ the induced conformal structure on its solutions space is given by $[\eta^1\eta^4-\eta^2\eta^3]$ where
  \begin{equation}
    \label{eq:etas-ODE-conformal}
    \eta^1=\exd Y-\half(F_{Y}\exd y+F_P\exd p),\quad \eta^2=\exd P-\half(G_Y\exd y+G_P\exd p),\quad \eta^3=\exd y,\quad \eta^4=\exd p,
      \end{equation}
pulled-back to $x=\mathrm{const}.$ Using the expressions above, setting $y'=Y$ and $p'=P$ in \eqref{eq:chains-flat-2D-path} and adapting the coframe further to the para-complex structure, one obtains that an adapted coframe is given by
\begin{equation}
  \label{eq:adapted-coframe-pKE}
\eta^1 = \exd Y,\quad \eta^2 = \tfrac{1}{\Delta^4}(-\Delta P\exd Y+\Delta^2\exd P-P^2\exd y+2\Delta P\exd p),\quad  \eta^3 = \exd y, \quad \eta^4 =\tfrac{1}{\Delta^3}( -P\exd y+\Delta\exd p)
\end{equation}
where $\Delta=Y-p.$  In this adapted coframe the  pseudo-Riemannian Einstein metric and the symplectic 2-form for the para-K\"ahler-Einstein structure are $\eta^1\eta^4-\eta^2\eta^3$ and $\eta^1\w\eta^4+\eta^2\w\eta^3$. The integrable rank 2 null distributions corresponding  to  the so-called  \emph{para-holomorphic} and \emph{anti-para-holomorphic} distributions  are $\ker\{\eta^1,\eta^3\}$ and $\ker\{\eta^2,\eta^4\}.$

Similar to our discussion above on the para-K\"ahler-Einstein metric on the space of flat 3D para-CR chains, the space of chains on the 3-sphere  $\SSS^3\subset\CC^2$ is the homogeneous space $\mathrm{SU}(2,1)\slash\mathrm{U}(1,1),$ which can be identified with the complement of the a closed ball in $\CC\PP^2.$   As is clear from Proposition \ref{prop:chains-3d-cr-necessary-cond}, the path geometry of chains of a CR structure is torsion-free if and only if the CR structure if flat. The induced structure on the space of chains for the flat CR structure on $\SSS^3$ is the indefinite Fubini-Study metric. Mimicking the derivation in of the coframe \eqref{eq:adapted-coframe-pKE}, an adapted coframe for this K\"ahler-Einstein metric is given by 
\[  \eta^1=\tfrac{1}{\Delta}\exd Y-\tfrac{Y}{\Delta^2}\exd P-\tfrac{Y(Y^2+3)}{2\Delta^2}\exd y+\tfrac{(1+Y^2)^2}{2\Delta^3}\exd p,\ \ 
  \eta^2=\tfrac{1}{\Delta^2}(\exd P-\tfrac{3 Y^2+1}{2} \exd y),\ \  \eta^3=\exd y-\tfrac{Y}{\Delta}\exd p,\ \ \eta^4=\tfrac{1}{\Delta}\exd p\]
where $\Delta=P-y.$ With respect to this coframe, the pseudo-Riemannian Einstein metric and the symplectic 2-form for the K\"ahler-Einstein structure are given by $\eta^1\eta^4-\eta^2\eta^3$ and $\eta^1\w\eta^4+\eta^2\w\eta^3.$ The holomorphic and anti-holomorphic sub-bundles of the complexified tangent bundle are $\ker\{\eta^1+\ri\eta^2,\eta^3+\ri\eta^4\}$ and $\ker\{\eta^1-\ri\eta^2,\eta^3-\ri\eta^4\},$ respectively.

\subsection{Characterization of chains: para-CR 3-manifolds}
\label{sec:path-geom-aris-1}
 In this section we present a way of determining whether a 3D path geometry arises as the chains of a para-CR geometry. We first note that by \eqref{eq:TC-3D-to-2D}, a necessary condition for such 3D path geometries is that the curvature $\bC$ has two distinct real roots of multiplicity 2. We use the following proposition in order to describe our characterization.
\begin{proposition}\label{prop:type-D-curv-chains}
  Given a 3D path geometry $(Q,\scX,\scV)$ with associated Cartan geometry $(\cG\to Q,\psi),$ where $\psi$ is given as \eqref{eq:path-geom-cartan-conn-3D}, if the curvature $\bC$ has 2 distinct real roots of multiplicity 2, then there is a principal $B$-subbundle $\iota\colon\cG_{D_r}\hookrightarrow\cG,$  where  $B\subset\mathrm{GL}(2,\RR)$ is the Borel subgroup, over which 
  \[\iota^* W_0=\iota^*W_1=\iota^*W_3=\iota^* W_4=0,\quad \iota^*W_2=\pm 1,\]
  and the components of $\iota^*\psi$, as given in \eqref{eq:path-geom-cartan-conn-3D}, satisfy
  \[
    \begin{aligned}
      \iota^*\psi^1_2\equiv& 0&&\mod \{\iota^*\alpha^0,\iota^*\alpha^2,\iota^*\beta^2\},\\
      \iota^*\psi^2_1\equiv& 0&&\mod \{\iota^*\alpha^0,\iota^*\alpha^1,\iota^*\beta^1\},\\
      \iota^*\psi^2_2\equiv& -\iota^*\psi^1_1&&\mod \{\iota^*\alpha^0,\iota^*\alpha^1,\iota^*\alpha^2,\iota^*\beta^1,\iota^*\beta^2\}\\
      \iota^*\nu_1,\iota^*\mu_1,\iota^*\nu_2,\iota^*\mu_2\equiv & 0&&\mod  \{\iota^*\alpha^0,\iota^*\alpha^1,\iota^*\alpha^2,\iota^*\beta^1,\iota^*\beta^2\}.
    \end{aligned}
  \]
  The 2-form
  \begin{equation}
    \label{eq:rho-GDr}
    \rho=\alpha^1\w\beta^2+\alpha^2\w\beta^1\in\Omega^2(\cG_{D_r})
  \end{equation}
  defines an invariant 2-form on $Q.$
\end{proposition}
\begin{proof}
  The proof is done via a standard application of Cartan's reduction procedure in the following way. Recall that a 3D path geometry is a Cartan geometry of type $(\mathrm{PSL}(4,\RR),P)$ where $P=P_0\ltimes P_+$ is the stabilizer of a flag of a line inside a plane in $\RR^4$,  $P_0=\RR^*\times\mathrm{GL}(2,\RR)$ is the reductive subgroup of $P,$ referred to as the structure group, and $P_+$ is the nilpotent normal subgroup of $P$. As was mentioned before, the curvature of a 3D path geometry, $\bC,$ can be presented as the quartic \eqref{eq:quadric-quartic} with an induced  $\mathrm{GL}(2,\RR)$-action by the structure group. Parametrically,  the structure group $P_0$ is a  block-diagonal matrix expressed as
  \begin{equation}
    \label{eq:P-0-3Dpath}
    P_0=\left\{A\in\mathrm{PSL}(4,\RR)\ \vline \ A=\mathrm{diag}(\tfrac{1}{a_{00}\det(\bH)},a_{00},\bH),\bH=
    \begin{pmatrix}
      a_{11} & a_{12}\\
      a_{21} & a_{22}
    \end{pmatrix}
  \right\}
\end{equation}
and one has 
\begin{equation}
  \label{eq:P-p-3Dpath}
P_+=\left\{B\in\mathrm{PSL}(4,\RR)\ \vline\  B=
    \begin{pmatrix}
      \hspace{-.3cm}1 & \hspace{-.3cm} p_0 & p_1+\half p_0q_1 & p_2+\half p_0q_2 & \\
      \hspace{-.3cm}0 & \hspace{-.3cm} 1 & q_1 & q_2 &\\
      \ \ 0_{2\times 1} & \ \ 0_{2\times 1} & &\hspace{-1.9cm} \mathrm{Id}_{2\times 2}&\\
    \end{pmatrix}
\right\}.\end{equation}
Since it is assumed that $\bC$ has two distinct real roots, using the action of $\mathrm{GL}(2,\RR),$ it is possible to translate the real roots to  $0$ and $\infty.$ More explicitly, for a choice of trivialization of $\cG,$ let
$\bC(u)=\Sigma_{i=0}^4
{\tiny\begin{pmatrix}
  4\\ i
\end{pmatrix}}
W_i(u)(\beta^1)^{4-i}(\beta^2)^i$
be the curvature at $u\in\cG.$ Using the right action  of the fibers, for  $g\in P$  and $u\in\cG$ the equivariant transformation of the Cartan connection and Cartan curvature under the gauge transformation implies
\[\psi(u)\to \psi(r_gu)=g^{-1}\psi(u)g+g^{-1}\exd g,\quad \Psi(u)\to\Psi(r_gu)= g^{-1}\Psi(u)g.\]
Consequently,    using the group parameters in  \eqref{eq:P-0-3Dpath} and \eqref{eq:P-p-3Dpath} to express $g\in P$, it is straightforward to obtain 
\begin{equation}
  \label{eq:action-W0-W4}
    \begin{aligned}
      W_0(g^{-1}u)=&W_0(u)a_{11}^4+4W_1(u)a_{11}^3a_{21}+6W_2(u)a_{11}^2a_{21}^2+4W_3(u)a_{11}a_{21}^3+W_4(u)a_{21}^4,\\
      W_4(g^{-1}u)=&W_0(u)a_{12}^4+4W_1(u)a_{12}^3a_{22}+6W_2(u)a_{12}^2a_{22}^2+4W_3(u)a_{12}a_{22}^3+W_4(u)a_{22}^4.
  \end{aligned}
  \end{equation}
  Take $g\in P$ such that $a_{11}=a_{22}=1.$ Since $\bC(u)$ has two distinct real roots, by \eqref{eq:action-W0-W4}, group parameters  $a_{12}$ and $a_{21}$ can be chosen so that  at $w=g^{-1}u\in\cG$ one has $W_0(w)=W_4(w)=0.$ Since both roots have multiplicity two and are distinct, it follows that $W_1(w)=W_3(w)=0$ and $W_2(w)\neq 0.$  As a result, one can define a sub-bundle $\iota_1\colon\cG^{(1)}\hookrightarrow\cG$ characterized by 
  \begin{equation}
    \label{eq:G1-first-reduction-curvature-chains)}
    \cG^{(1)}=\left\{u\in\cG\ \vline \ W_0(u)=W_1(u)=W_3(u)=W_4(u)=0\right\}.
      \end{equation}
      By our discussion above, the bundle  $\cG^{(1)}\to Q$ is a principal $P^{(1)}$-bundle where $P^{(1)}=(\RR^*)^3\ltimes P^+$ and $(\RR^*)^3\subset P_0$ is the Cartan subgroup, given by setting $a_{12}=a_{21}=0$ in \eqref{eq:P-0-3Dpath}. As a result, one obtains that $\iota^*_1\bC=6W_2(\beta^1)^2(\beta^2)^2,$ where by abuse of notation we have suppressed $\iota_1^*$ on the right hand side.  Moreover, the pull-back of   the Bianchi identities for $\exd W_3$ and $\exd W_1,$ given in \eqref{eq:W-A-curvature-torsion-Bianchies}, to $\cG^{(1)}$ gives $\iota_1^*\psi^2_1\equiv 0$ and $\iota_1^*\psi^1_2\equiv 0$ modulo $\{\iota^*_1\alpha^0,\iota^*_1\alpha^1,\iota^*_1\alpha^2,\iota^*_1\beta^1,\iota^*_1\beta^2\}.$ Suppressing  $\iota^*_1,$ one can write
      \begin{equation}
        \label{eq:psi12-psi21-chains}
        \psi^1_2=A^1_{2i}\alpha^i+B^1_{2a}\beta^a,\qquad \psi^2_1=A^2_{1i}\alpha^i+B^2_{1a}\beta^a,
      \end{equation}
      for some functions $A^a_{bi}$ and $B^a_{bc}$ on $\cG^{(1)}.$

      Since the curvature $\bC$ has to have two distinct real roots, it follows that $W_2\neq 0$ on $\cG^{(1)}.$ The action of $P^{(1)}$ on $W_2$ is given by
      \begin{equation}
        \label{eq:W2-group-action}
        W_2(g^{-1}u)=a_{11}^2a_{22}^2W_2(u).
      \end{equation}
      Thus, depending on the sign of $W_2,$ one can normalize it to $\pm 1.$   From now on we assume   $W_2> 0$ since the case $W_2<0$ can be treated identically.   See Remark \ref{rmk:sign-of-W_2-chains} for the difference of outcome in these two cases.  Define a sub-bundle $\iota_2\colon\cG^{(2)}\hookrightarrow \cG^{(1)}$ as
      \begin{equation}
        \label{eq:P2-2nd-reduction}
        \cG^{(2)}=\left\{u\in\cG^{(1)}\ \vline\ W_2(u)=1\right\}.
      \end{equation}
      It follows that $\cG^{(2)}\to Q$ is a principal $P^{(2)}$-bundle where $P^{(2)}=(\RR^*)^2\ltimes P_+$ and $(\RR^*)^2\subset P_0$ is given by $a_{12}=a_{21}=0$ and $a_{22}=1/a_{11}$ in \eqref{eq:P-0-3Dpath}. Using the Bianchi identities \eqref{eq:W-A-curvature-torsion-Bianchies}, via pull-back to $\cG^{(2)},$ one obtains
      \begin{equation}
        \label{eq:psi22-reduced-chains}
        \psi^2_2=-\psi^1_1+A^2_{2i}\alpha^i+B^2_{2a}\beta^a
      \end{equation}
      for functions $A^2_{2i}$ and $B^2_{2a}$ on $\cG^{(2)}.$

      The pull-back of the Cartan connection $\psi$ to $\cG^{(1)}$ and $\cG^{(2)}$ is no longer equivariant under the action of the fibers $P^{(1)}$ and $P^{(2)}$, respectively. Nevertheless, the pull-back of $\psi$ to them defines a trivialization of the tangent bundle  $T\cG^{(2)}$, i.e. a so-called \emph{$\{e\}$-structure}. It is straightforward to find the action of the fibers on the quantities $A^a_{bk}$ and $B^a_{bc}.$ In particular, using the parametrization in \eqref{eq:P-p-3Dpath}, an action by $g\in P^{(2)}$ gives
      \begin{equation}\label{eq:AB-group-action-chains}
        \begin{aligned}
          B^1_{21}(g^{-1}u)=&\tfrac{1}{a_{11}a_{00}}B^1_{21}(u)+q_2\\
          B^2_{12}(g^{-1}u)=&\tfrac{a_{11}}{a_{00}}B^2_{12}(u)+q_1\\
          A^1_{21}(g^{-1}u)=&{\tfrac{a_{00}^2}{a_{11}^2}q_1A^1_{20}(u)+\tfrac{a_{00}}{a_{11}}A^1_{21}(u)-\tfrac{1}{a_{11}a_{00}}p_0B^{1}_{21}(u) -\half p_0q_2+p_2}\\
          A^2_{12}(g^{-1}u)=& a_{00}^2a_{11}^2q_2A^2_{10}(u)+a_{00}a_{11}A^2_{12}(u)-\tfrac{a_{11}}{a_{00}}p_0B^2_{12}(u)-\half p_0q_1+p_1
        \end{aligned}
      \end{equation}
      Infinitesimally, these actions correspond to Bianchi identities 
      \begin{equation}
        \label{eq:AB-inf-action-Bianchies-chains}
        \begin{aligned}
          \exd B^1_{21}\equiv& -(\psi^0_0+\psi^1_1)B^1_{21}+\nu_2\\
          \exd B^2_{12}\equiv& -(\psi^0_0-\psi^1_1)B^2_{12}+\nu_1\\
          \exd A^1_{21}\equiv&  \nu_1A^1_{20}+(\psi^0_0-\psi^1_1)A^1_{21}-B^1_{21}\mu_0+\mu_2\\
          \exd A^2_{12}\equiv&  \nu_2A^2_{10}+(\psi^0_0+\psi^1_1)A^2_{12}-B^2_{12}\mu_0+\mu_1
        \end{aligned}
      \end{equation}
      modulo $\{\alpha^0,\alpha^1,\alpha^2,\beta^1,\beta^2\}.$ As a result, the sub-bundle $\iota_3\colon\cG^{(3)}\hookrightarrow \cG^{(2)}$ given by
      \begin{equation}
        \label{eq:G3-reduced-Pp-chains}
        \cG^{(3)}=\{u\in\cG^{(2)}\ \vline\ B^1_{21}(u)=B^2_{12}(u)=A^1_{21}(u)=A^2_{12}(u)=0\}
      \end{equation}
      is well-defined as a principal $B$-bundle where $B\cong(R^*)^2\ltimes\RR\subset \mathrm{GL}(2,\RR)$ is the Borel subgroup. In terms of parametrizations \eqref{eq:P-0-3Dpath} and \eqref{eq:P-p-3Dpath} for $P=P_0\ltimes P_+,$ one can express $B\subset P$ as $a_{22}=1/a_{11}$ and $a_{12}=a_{21}=p_1=p_2=q_1=q_2=0.$

      In the differential relations \eqref{eq:AB-inf-action-Bianchies-chains}, the pull-back for the first two equations to $\cG^{(3)}$ imply $\nu_1,\nu_2$ vanish mod $\{\alpha^i,\beta^a\}.$ Consequently, the last two relations  imply $\mu_1$ and $\mu_2$ vanish mod $\{\alpha^i,\beta^a\}.$ The reduction of $\mu_1,\mu_2,\nu_1,\nu_2$ on $\cG^{(3)},$ together with the pull-back of \eqref{eq:psi12-psi21-chains} and \eqref{eq:psi22-reduced-chains} to $\cG^{(3)}$ finishes the proof of the first part, where $\iota:=\iota_1\circ\iota_2\circ\iota_3$ and $\cG_{D_r}:=\cG^{(3)}.$

      Lastly, one can check that $\rho$ in \eqref{eq:rho-GDr} is well-defined on $\cG_{D_r}$ and is invariant under the action of the fibers of  $\cG_{D_r}\to Q.$ Thus, it defines an invariant 2-form on $Q.$
    \end{proof}
     \begin{remark}\label{rmk:D-r-curvature-type}
   A basic invariant of a binary quartic, such as the curvature $\bC,$ acted on by $\mathrm{GL}(2,\RR),$ is its root type. Motivated by the Petrov classification of the self-dual and anti-self-dual Weyl curvature of a Lorentzian conformal structure, one finds 10 possible algebraic types for the quartic $\bC$ in our setting depending on the multiplicity and reality of the root. 
Motivated by the symbols used for Petrov types,  when a quartic has two distinct real roots of multiplicity two its algebraic type is denoted by $D_r,$ hence we denote the  reduced 8-dimensional bundle in this case by $\cG_{D_r}.$  
\end{remark}
    \begin{remark}
      In the proof of  Proposition \ref{prop:type-D-curv-chains} it was not necessary to know the  explicit group actions as given in \eqref{eq:action-W0-W4}, \eqref{eq:W2-group-action}, and  \eqref{eq:AB-group-action-chains}. We provided the explicit form in order to clarify the reduction procedure. In order to carry out such reductions it suffices to have the infinitesimal form of the group action on invariants which, as mentioned above, on $\cG$ and $\cG^{(2)}$  are given by the Bianchi identities \eqref{eq:W-A-curvature-torsion-Bianchies}, and \eqref{eq:AB-group-action-chains}, respectively. We refer the reader to \cite{Gardner-Book} for a discussion on the relation between explicit group action and its infinitesimal form and also for the notion of an $\{e\}$-structure in the context of Cartan's method of equivalence, which appeared in the proof above. 
    \end{remark}

Before proving the main theorem in the para-CR case, we prove the following intermediate theorem which identifies a natural class of 3D path geometries that contains chains as a proper subclass. 
\begin{theorem}\label{thm:2d-path-geometries-generalized-chains}
Let $(\cG\to Q,\psi)$ be the Cartan geometry associated with a 3D path geometry $(Q,\scX,\scV)$ satisfying the following conditions:
  \begin{enumerate}
  \item The quartic $\bC$ has two distinct real roots of multiplicity 2. 
  \item The invariantly defined 2-form $\rho=\alpha^1\w\beta^2+\alpha^2\w\beta^1\in\Omega^2(Q)$   from Proposition \ref{prop:type-D-curv-chains} is closed.
  \end{enumerate}
  Then the Pfaffian systems $\cI_2:=\{\alpha^0,\alpha^1\}$ and $\cI_1=\{\alpha^0,\alpha^2\}$ are integrable and the 3D leaf space of $\{\alpha^0,\alpha^1,\alpha^2\},$ denoted by $N,$ is equipped with a para-CR structure.  The projection of each path on $Q$ to $N$ is transverse to the contact distribution  $\ker\alpha^0\subset TN.$ The invariants $P_1$ and $Q_1$ of such  para-CR geometries  $(N,\scD_1,\scD_2)$ depend on the 4th jet of torsion entries $A_0$ and $A_2$ of the 3D path geometry, respectively. 
\end{theorem}
\begin{proof}
  By Proposition \ref{prop:type-D-curv-chains}, from condition (1) one obtains a sub-bundle $\iota\colon\cG_{D_r}\hookrightarrow\cG$ which is a principal $B$-bundle  $\cG_{D_r}\to Q.$

By Proposition \ref{prop:type-D-curv-chains},   the condition $\exd\rho=0$ is invariant and implies the vanishing conditions
  \[A_1=A^1_{20}=A^2_{10}=A^2_{20}=A^2_{21}=A^2_{22}=B^2_{21}=B^2_{22}  = 0,\]
for the functions $A^i_{jk}$ and $B^a_{bi}$  on $\cG_{D_r}$ defined  in \eqref{eq:psi12-psi21-chains} and \eqref{eq:psi22-reduced-chains}. In particular, by \eqref{eq:psi22-reduced-chains}, one has $\psi^2_{2}=-\psi^1_1.$ Checking the differential consequences of this vanishing condition is a matter of tedious computation which yields 
  \begin{equation}
    \label{eq:quasi-symp-consequences-chains}
    \begin{gathered}
      \psi^1_2=0,\quad \psi^2_1=0,\quad \psi^2_2=-\psi^1_1,\quad \nu_1=\half\alpha^2,\quad \nu_2=-\half\alpha^1,\\ \mu_1=-\half\beta^2-\tfrac{1}{4}A_{0;\underline{22}}\alpha^2-\tfrac{1}{4}A_{0;\underline{2}}\alpha^0,\quad \mu_2=\half\beta^1+\tfrac{1}{4}A_{2;\underline{11}}\alpha^1+\tfrac{1}{4}A_{2;\underline{1}}\alpha^0.
    \end{gathered}
  \end{equation}
  where the pull-back $\iota^*$ is suppressed.  We refer the reader to \ref{sec:conventions} for the definition of some of the coefficients appearing  in \eqref{eq:quasi-symp-consequences-chains}, i.e. first and second coframe derivatives of $A_0$ and $A_2$. Note that the relations $A_{0;\underline{22}}=-A_{2;\underline{11}}$  and $A_{0;\underline{21}}=A_{2;\underline{21}}=0$ follow from $\exd\rho=0$.

  Using the relations \eqref{eq:quasi-symp-consequences-chains} to compute the Cartan curvature $\Psi$ for the 3D path geometry on $\cG_{D_r}$, it follows that on $\cG_{D_r}$ one has
  \begin{equation}
    \label{eq:general-chain-2D-path}
    \begin{aligned}
      \exd\alpha^0=&-2\psi^0_0\w\alpha^0+\alpha^1\w\alpha^2,\\
      \exd\alpha^1=&(-\psi^0_0-\psi^1_1)\w\alpha^1-\beta^1\w\alpha^0,\\
      \exd\alpha^2=&(-\psi^0_0+\psi^1_1)\w\alpha^2-\beta^2\w\alpha^0.
    \end{aligned}
  \end{equation}
  Thus, the Pfaffian systems $\cI_2:=\{\alpha^0,\alpha^1\}$ and $\cI_1:=\{\alpha^0,\alpha^2\}$ are integrable. Moreover, by \eqref{eq:general-chain-2D-path}, the leaf space of $\{\alpha^0,\alpha^1,\alpha^2\}$ denoted by $N,$ defines a 3D para-CR geometry on $M$ with contact distribution $\ker\alpha^0=\scD_1\oplus\scD_2$ where $\scD_1:=\ker\cI_1=\langle\tfrac{\partial}{\partial\alpha^1}\rangle$ and $\scD_2:=\ker\cI_2=\langle\tfrac{\partial}{\partial\alpha^2}\rangle.$ Furthermore, using the quotient map $\nu\colon Q\to N,$ it is clear that the tangent line to paths on $Q,$ i.e. $\langle\tfrac{\partial}{\partial\alpha^0}\rangle,$ are mapped to  lines that are transverse to the contact distribution via $\nu_*.$

More precisely,  it follows that  the Cartan geometry for the 3D para-CR geometry on $N$ is given by $(\cG_{D_r}\to N,\phi)$ where
\begin{equation}
  \label{eq:phi-2D-path-general-chain}
  \phi=    \begin{pmatrix} 
      \tfrac{1}{6} A_{0;\underline{22}}\alpha^0-\psi^0_0-\tfrac 13\psi^1_1 & B_{20}\alpha^0-\beta^2 & \mu_0+B_{0i}\alpha^i\\
      \alpha^1 & -\tfrac{1}{12}A_{0;\underline{22}}\alpha^0+\tfrac 23\psi^1_1 & B_{10}\alpha^0+B_{11}\alpha^1+\beta^1  \\
      \alpha^0 & \alpha^2 & -\tfrac{1}{12}A_{0;\underline{22}}\alpha^0+\psi^0_0-\tfrac 13\psi^1_1   
    \end{pmatrix}
    \end{equation}
in which
  \[
    \begin{gathered}
      B_{10}=\tfrac{7}{36}A_{2;\underline{11}2}-\tfrac 19A_{2;\underline{1}2\underline{1}},\quad B_{11}=-\tfrac{1}{4}A_{0;\underline{22}},\quad B_{20}=\tfrac{7}{36}A_{0;\underline{22}1}-\tfrac 19A_{0;\underline{2}1\underline{2}},\\
      B_{02}=\tfrac{5}{36}A_{2;\underline{11}2}+\tfrac{1}{36}A_{2;\underline{1}2\underline{1}},\quad B_{01}=\tfrac{1}{36}A_{0;\underline{2}1\underline{2}}-\tfrac{1}{9}A_{0;\underline{22}1},\quad B_{00}=\tfrac{7}{36}A_{0;\underline{22}12}-\tfrac{1}{9}A_{0;\underline{2}1\underline{2}2}.            
    \end{gathered}
  \]
  Consequently, the invariants $P_1$ and $Q_1$ for such path geometries are given by
  \begin{equation}
    \label{eq:TU-general-chain}
    P_1=\tfrac{7}{36}A_{0;\underline{22}11}-\tfrac 19A_{0;\underline{2}1\underline{2}1}+A_0,\qquad Q_1=\tfrac{7}{36}A_{2;\underline{11}22}-\tfrac 19A_{2;\underline{1}2\underline{1}2}+A_2
  \end{equation}        
\end{proof}

\begin{remark}\label{rmk:3d-path-geometries-chain-general}
  The 3D path geometry obtained in Theorem \ref{thm:2d-path-geometries-generalized-chains}  is an example of  variational orthopath structures defined in \cite{MS-cone}. This is due to the fact that in the path geometry $(Q,\scX,\scV)$  the conformal class of the bundle metric $[s^*\beta^1\circ s^*\beta^2]\subset \mathrm{Sym}^2(\scV^*)$ is well-defined for any section $s\colon Q\to\cG_{D_r}$  and the 2-form $\rho$ is a  \emph{compatible} quasi-symplectic 2-form, i.e. $\rho\w\rho\neq 0,\exd\rho=0,$ the paths are characteristic curves of $\rho,$ and the fibers of $Q\to N$ are isotropic. It is shown in \cite{MS-cone} that the paths of such structures are the extremal curves of a class of non-degenerate first order Lagrangians. Using Cartan-K\"ahler analysis, one can find  the local generality of real analytic 3D path geometries satisfying the conditions in Theorem \ref{thm:2d-path-geometries-generalized-chains}. It turns out that their local generality is  3 functions of 3 variables. 
\end{remark}
\begin{remark}\label{rmk:sign-of-W_2-chains}
  The sign of $W_2$ determines the orientation induced on the 3D para-CR geometry $(N,\scD_1,\scD_2)$ from the 3D path geometry $(Q,\scX,\scV)$ where $\scD_i=\tau_*\langle\tfrac{\partial}{\partial\alpha^i}\rangle$. More precisely, before normalizing $W_2$ to $\pm 1,$ one has $\exd\alpha^0\equiv W_2\alpha^1\w\alpha^2$ mod $\{\alpha^0\}.$ Thus, when $W_2>0$ it follows that $\alpha^1$ and $\alpha^2$ are $\scD_1$-positive and $\scD_2$-positive, respectively, and $\alpha^0$ is positive with respect to the induced co-orientation on the contact distribution $\scC=\scD_1\oplus\scD_2$, as we recalled  at the end of \ref{sec:cr-3-manifolds}. Similarly, it follows that  when $W_2<0$ then  $\exd\alpha^0$ in \eqref{eq:general-chain-2D-path} changes to $\exd\alpha^0\equiv-\alpha^1\w\alpha^2$ modulo $\{\alpha^0\}.$ Thus, the descent from the 3D path geometry to  the 3D para-CR geometry, as described above, induces a negative co-orientation. In   \cite[Corollary 4.4]{CZ-CR} the relation between the sign  of $W_2$ and the  orientation induced on the 3D para-CR geometry  is described by the fact that chain preserving contact diffeomorphisms induce an automorphism or an anti-automorphism of the underlying structure, which in our case is either a 3D para-CR or CR structure. 
\end{remark}

Note that by \eqref{eq:TU-general-chain},  3D path geometries in Theorem \ref{thm:2d-path-geometries-generalized-chains} do not satisfy the necessary condition \eqref{eq:TC-3D-to-2D} relating the torsion entries of the 3D path geometry to $P_1$ and $Q_1.$ It turns out that adding this necessary condition to conditions (1) and (2) in Theorem \ref{thm:2d-path-geometries-generalized-chains} is also sufficient for a 3D path geometry to arise as chains of a  para-CR geometry.
\begin{proof}[Proof of Theorem \ref{thm:3D-path-CR-para-CR}: the para-CR case]
As was discussed in \ref{sec:an-overview-chains}, conditions (1) and (2) are necessary conditions  for a 3D path geometry to arise as chains. Furthermore, if the initial path geometry corresponds to the chains of the resulting para-CR structure then by Proposition \ref{prop:chains-2D-path-necessary-cond} the necessary condition \eqref{eq:TC-3D-to-2D}  has to hold on $\cG_D:=\cG_{D_r}$  as well which implies that, restricted to $\cG_D,$  the torsion entries $A_0$ and $A_2$ need to be well-defined up to scale on $N,$ i.e.
  \[\exd A_0,\exd A_2\equiv 0\mod\{\alpha^0,\alpha^1,\alpha^2,\psi^0_0,\psi^1_1\}.\]
  Using \eqref{eq:quasi-symp-consequences-chains} and \eqref{eq:phi-2D-path-general-chain}, condition (3) implies 
  \[       \def\arraystretch{1.3}
\iota^*\psi=   \begin{pmatrix} 
      -\psi^0_0 & \mu_0 & -\half \beta^2 & \half \beta^1\\
      \alpha^0 & \psi^0_0 & \half\alpha^2 & -\half\alpha^1\\
      \alpha^1 & \beta^1 & \psi^1_1 &  0\\
      \alpha^2 & \beta^2 & 0 & -\psi^1_1
    \end{pmatrix},\qquad   \phi=    \begin{pmatrix} 
      -\psi^0_0-\tfrac 13\psi^1_1 & -\beta^2 & \mu_0\\
      \alpha^1 & \tfrac 23\psi^1_1 & \beta^1  \\
      \alpha^0 & \alpha^2 & \psi^0_0-\tfrac 13\psi^1_1   
    \end{pmatrix}.
\]
 where $(\cG_{D}\to N,\phi)$ is the Cartan geometry for the 3D para-CR geometry  induced by $\iota^*\psi$ and, by \eqref{eq:TU-general-chain}, it follows that the invariants $P_1$ and $Q_1$ are $\iota^*A_0$ and $\iota^*A_2,$ respectively.  We recall that, as was explained in \ref{sec:conventions}, in our notation condition (3) can be expressed as $A_{0;\underline{2}}=A_{2;\underline{1}}=0.$ This is due to the fact that conditions $A_{0;\underline{1}}=A_{2;\underline{2}}=0 $ already follow from conditions (1) and (2) and are satisfied for 3D path geometries in Theorem \ref{thm:2d-path-geometries-generalized-chains}.
 
Furthermore, the resulting Cartan connection $\phi$  uniquely determines $\iota^*\psi$ which coincides with what is obtained via the extension functor from chains of the 3D para-CR geometry $(\cG_{D}\to N,\phi)$ as discussed in \ref{sec:an-overview-chains}. Thus,  conditions (1),(2),(3) provide necessary and sufficient conditions for a 3D path geometry to arise as chains of a 3D para-CR geometry.
\end{proof}
\begin{remark}\label{rmk:chains-3d-path-geometries-exclusive}
  Given a pair of second order ODEs, checking conditions (1),(2) and (3) only involves finding  roots of a quartic, linear algebra  and differentiation and can be verified straightforwardly.    Note that the line  fields spanned by the vector fields $\tfrac{\partial}{\partial\alpha^1},\tfrac{\partial}{\partial\alpha^2}\tfrac{\partial}{\partial\beta^1},\tfrac{\partial}{\partial\beta^2}$ are well-defined on $\cG_{D_r}$ and, therefore, condition (3) is easy to verify. Equivalently, using Remark \ref{rmk:R-T-discriminant}, condition (3) can be given as expressing the absolute CR invariant $\bR$ in terms of  $\Delta_{\bT_\cP}$ and checking that it is invariant with respect to the action of   $B,$ i.e. the fibers of $\cG_D\to N.$ Also, in the proof above we changed the subscript $D_r,$ which  reflects the root type according to Remark \ref{rmk:D-r-curvature-type},  to $D$ so that it is consistent with the statement of Theorem \ref{thm:3D-path-CR-para-CR}.

  Lastly, we point out that the curvature of  path geometries in dimensions larger than 3 cannot be represented as a binary polynomial. Nevertheless,  one can find a replacement for condition (1) in Theorem \ref{thm:3D-path-CR-para-CR} for the curvature of path geometries defined by chains of higher dimensional (para-)CR structures. However, condition (3)  is never true for chains of   non-flat (para-)CR structures in higher dimensions. 
\end{remark}

\subsection{Characterization of  chains: CR 3-manifolds}   
\label{sec:3d-path-geometries}

In this section we use the same strategy as in the previous section to characterize the path geometry of CR chains in dimension three. We start by an analogue of Proposition \ref{prop:type-D-curv-chains} in the CR setting. 
  \begin{proposition}\label{prop:type-D-curv-CR-chains}
  Given a 3D path geometry $(Q,\scX,\scV)$ with associated Cartan geometry $(\cG\to Q,\psi),$ where $\psi$ is given as \eqref{eq:path-geom-cartan-conn-3D}, if the curvature $\bC$ has a  non-real complex root of multiplicity 2, then there is a principal $B$-subbundle $\iota\colon\cG_{D_c}\hookrightarrow\cG,$  where  $B\subset\mathrm{U}(1,1)$ is the Borel subgroup, over which 
  \[\iota^*W_1=\iota^*W_3=0,\quad \iota^* W_2=\tfrac 13\iota^* W_0=\tfrac 13\iota^*W_4=\pm 1,\]
  and the components of $\iota^*\psi$, as given in \eqref{eq:path-geom-cartan-conn-3D}, satisfy
  \begin{equation}
    \label{eq:reduc-relat-CR}
        \begin{aligned} 
      \iota^*\psi^2_1\equiv& -\iota^*\psi^1_2&&\mod \{\iota^*\alpha^0,\iota^*\alpha^1,\iota^*\alpha^2,\iota^*\beta^1,\iota^*\beta^2\},\\
      \iota^*\psi^1_1\equiv& 0&&\mod \{\iota^*\alpha^0,\iota^*\alpha^2,\iota^*\beta^2\},\\
      \iota^*\psi^2_2\equiv& 0 &&\mod \{\iota^*\alpha^0,\iota^*\alpha^1,\iota^*\beta^1\},\\
      \iota^*\nu_1,\iota^*\mu_1,\iota^*\nu_2,\iota^*\mu_2\equiv & 0&&\mod  \{\iota^*\alpha^0,\iota^*\alpha^1,\iota^*\alpha^2,\iota^*\beta^1,\iota^*\beta^2\}.
    \end{aligned}
    \end{equation}
  The 2-form
  \begin{equation}
    \label{eq:rho-GDr}
    \rho=\alpha^1\w\beta^1+\alpha^2\w\beta^2\in\Omega^2(\cG_{D_c})
  \end{equation}
  defines an invariant 2-form on $Q.$
\end{proposition}
\begin{proof}
  Since the proof is similar to that of Proposition \ref{prop:type-D-curv-chains}, we only highlight the differences. 
Since the binary quartic $\bC$ \eqref{eq:quadric-quartic} has a non-real complex root of multiplicity two,  it can be put in the form
  \[\bC=(A(\beta^1)^2+2B\beta^1\beta^2+C(\beta^2)^2)^2\otimes V\otimes X^{-1},\]
  where $B^2-AC<0$ for some functions $A,B,C$ on $\cG.$ Relating $A,B,C$ to $W_i$'s and using the parametrization of $P_0$  in \eqref{eq:P-0-3Dpath}, one can obtain the induced group action on $A-C$ and $B$  to be
  \begin{equation*}
    \label{eq:ABC-action}
    \begin{aligned}
     A(u)-C(u)&\to  A(r_gu)-C(r_gu)=A a_{11}^2-A a_{12}^2+2 B a_{11} a_{21}-2 B a_{12} a_{22}+C a_{21}^2-C a_{22}^2,\\
     B(u)&\to  B(r_gu)=Aa_{11} a_{12}+B a_{11} a_{22}+B a_{12} a_{21}+C a_{21} a_{22}.
    \end{aligned}
  \end{equation*}
Using the above relations, it is straightforward to find a set of group parameters at which $A-C$ and $B$ vanish. For instance, since with respect to the initial coframe the condition $AC> 0$ has to hold in order for  $B^2-AC<0$ to be satisfied, for the real-valued parameters $a_{21}=a_{22}=1,a_{11}=-\frac{B+\sqrt{AC-B^2}}{A}$ and $a_{12}=-\frac{B-\sqrt{AC-B^2}}{A},$ one obtains that  $A-C$ and $B$ vanish. Note that for such  values $\det(a_{ij})\neq 0$ remains valid. Thus, one can always find a choice of coframe with respect to which $A(r_gu)=C(r_gu)$ and $B(r_gu)=0,$ which gives $W_1(r_gu)=W_3(r_gu)=0$ and $W_0(r_gu)=W_4(r_gu)=3W_2(r_gu).$  As a result, if a 3D path geometry has a non-real complex root of multiplicity two, then one can define a  sub-bundle $\iota_1\colon\cG^{(1)}\hookrightarrow\cG$ characterized by
\[\cG^{(1)}=\left\{u\in\cG\,\vline\, W_1(u)=W_3(u)=0, W_0(u)=W_4(u)=3W_2(u)\right\}.\]
By our discussion above, $\cG^{(1)}\to Q$ is a principal $P^{(1)}$-bundle where $P^{(1)}=(\RR^*\times\mathrm{CO}(2))\ltimes P_+$.  More explicitly, using the group action relations above, the stabilizer of $B=0,A-C=0$ is given by
\[A(a_{11}a_{12}+a_{21}a_{22})=0,\quad A(a_{11}^2-a_{12}^2+a_{21}^2-a_{22}^2)=0,\]
which as  a subgroup of the structure group is $\RR^*\times\mathrm{CO}(2)\subset P_0.$ Thus, we can write $a_{11}=a_{22}=a\cos(b)$ and $a_{12}=-a_{21}=a\sin(b).$ Similarly to Proposition \ref{prop:type-D-curv-chains}, the action of $P^{(1)}$ on $W_2$ is given by
\[W_2(u)\to W_2(r_gu)=a^4 W_2(u).\]
Hence we can define a sub-bundle $\iota_2\colon\cG^{(2)}\hookrightarrow \cG^{(1)}$ defined as $\cG^{(2)}:=\{u\in\cG^{(1)}\,\vline\, W_2(u)=\pm 1\},$ as we did in \eqref{eq:P2-2nd-reduction}, which in this case is a principal $P^{(2)}$-bundle where  $P^{(2)}=(\R^*\times SO(2))\ltimes P_+.$ We assume $W_2>0,$ and refer to Remark \ref{rmk:sign-of-W_2-chains} when $W_2<0.$ Pulling-back the Bianchi identities \eqref{eq:W-A-curvature-torsion-Bianchies} to $\cG^{(2)}$, it follows that the reduced entries of $\psi$ in \eqref{eq:path-geom-cartan-conn-3D} are 
\[\iota_{12}^*\psi^1_2+\iota_{12}^*\psi^2_1,\ \iota_{12}^*\psi^1_1,\ \iota_{12}^*\psi^2_2\equiv 0\mod \iota^*_{12}\{\alpha^0,\alpha^1,\alpha^2,\beta^1,\beta^2\},\]
where $\iota_{12}=\iota_1\circ\iota_2\colon\cG^{(2)}\hookrightarrow \cG.$ For the third reduction, we proceed similarly to Proposition \ref{prop:type-D-curv-chains}, by considering the induced action of $P^{(2)}$ on $A^1_{11},B^{1}_{11},A^2_{22}$ and $B^2_{22}$ where $\iota_{12}^*\psi^j_k=A^j_{ki}\alpha^i+B^j_{ka}\beta^a.$ It is straightforward to follow the analogous step in the proof of Proposition \ref{prop:type-D-curv-chains} and show that the sub-bundle $\iota_3\colon\cG^{(3)}\to\cG^{(2)}$ defined as
\[\cG^{(3)}=\left\{u\in\cG^{(2)}\,\vline\, A^1_{11}(u)=A^2_{22}(u)=B^1_{11}(u)=B^2_{22}(u)=0\right\},\]
is a principal $B$-bundle where $B\subset \mathrm{U}(1,1)$ is the Borel subgroup. Consequently,  via pull-back by $\iota^*$ where $\iota:=\iota_1\circ\iota_2\circ\iota_3\colon\cG^{(3)}\to\cG,$ the relations  \eqref{eq:reduc-relat-CR} can be similarly shown to hold on $\cG_{D_c}:=\cG^{(3)}.$  Checking that the 2-form $\rho$ is invariant under the action of $B$ is also straightforward and is skipped. 
\end{proof}
Similar to the notation $\cG_{D_r}$ explained in Remark \ref{rmk:D-r-curvature-type}, the subscript $D_c$ for the principal bundle $\cG_{D_c}$ denotes the assumption that the quartic $\bC$ has a repeated non-real complex root of multiplicity two.

One can prove a statement similar to Theorem \ref{thm:2d-path-geometries-generalized-chains} by modifying conditions (1) and (2) to the setting of Proposition \ref{prop:type-D-curv-CR-chains}, e.g. a non-real complex root of multiplicity two . We leave that to the interested reader and  directly prove Theorem \ref{thm:3D-path-CR-para-CR} when $\bC$ has a non-real complex root of multiplicity two.
\begin{proof}[Proof of Theorem \ref{thm:3D-path-CR-para-CR}: the CR case] 
  Using Proposition \ref{prop:type-D-curv-CR-chains}, the  proof when $\bC$ has a non-real complex root of multiplicity two is almost identical to the one given in \ref{sec:path-geom-aris-1} wherein $\bC$ has two real roots.   Following the same steps as in the para-CR case, one needs to find the vanishing quantities in the $\{e\}$-structure  on $\cG_D:=\cG_{D_c}$ in Proposition \ref{prop:type-D-curv-CR-chains} that result from $\exd\rho=0$  together with  their differential consequences. Consequently, using the inclusion $\iota\colon\cG_D\hookrightarrow\cG,$ it follows that the 3D leaf space of $\{\alpha^0,\alpha^1,\alpha^2\},$ i.e. $N=Q\slash\scV,$ is equipped with a CR structure $(\cG_D\to N,\phi).$ However, as in Theorem \ref{thm:2d-path-geometries-generalized-chains}, in general the torsion of the path geometry $\bT_{\cG_D}$ and the fundamental invariant of the resulting CR structure  $R=R_1+\ri R_2$ are not  related as in Proposition \eqref{eq:TC-CR-to-3D}, or, equivalently, as in Remark \ref{rmk:R-T-discriminant}. As in the para-CR case,  this necessary condition between $\bT_{\cG_D}$ and $R$ is satisfied if and only if the coefficients of $\bT_{\cG_D}$ have no dependency on the fibers of $\cG_D\to N.$ 
  Conversely, using Proposition \ref{prop:chains-3d-cr-necessary-cond}, one knows that chains of any 3D CR geometry defines such path geometries. 
\end{proof}
\begin{remark}
  In \cite{Graham-CR,CG-CR} characterizations for   conformal structures arising from Fefferman's construction \cite{Fefferman-CR} for CR structures were given. The null geodesics of such conformal structures project to chains of the underlying CR structures. As mentioned in Remark \ref{rmk:chains-3d-path-geometries-exclusive}, unlike these characterizations, our characterization of the path geometry of chains are computationally verifiable and only holds in dimension three.
\end{remark}

\subsection*{Acknowledgments}

The authors  thank Maciej Dunajski for inspiring discussions and the anonymous referee for helpful comments. The authors also thank Dennis The for his comments which helped with deriving \eqref{eq:complex-point-trans}. WK was partially supported by the grant 2019/34/E/ST1/00188 from the National Science Centre, Poland. OM received funding from the Norwegian Financial Mechanism 2014-2021 with project registration number 2019/34/H/ST1/00636. OM  acknowledges partial support by the grant  PID2020-116126GB-I00 provided via the Spanish Ministerio de Ciencia e Innovaci\'on MCIN/ AEI /10.13039/50110001103. The  EDS calculations  are done using Jeanne Clelland's \texttt{Cartan} package in Maple.

\bibliographystyle{alpha}      
\bibliography{dancing}
\end{document}